\documentclass[11pt,reqno,tbtags]{amsart}
\usepackage[]{amsmath,amssymb,amsfonts,latexsym,amsthm,enumerate}
\usepackage{amssymb}

\numberwithin{equation}{section}

\newtheorem{conjecture}{Conjecture}
\newtheorem{theorem}{Theorem}[section]
\newtheorem*{theorem*}{Theorem}
\newtheorem{lemma}[theorem]{Lemma}

\newtheorem{proposition}[theorem]{Proposition}
\newtheorem{corollary}[theorem]{Corollary}

\newtheorem{condition}[theorem]{Condition}

\theoremstyle{definition}{

\newtheorem{definition}{Definition}

}

\theoremstyle{remark}{

\newtheorem*{remark*}{Remark}

}

\renewcommand{\epsilon}{\varepsilon}

\newcommand{\lh}{\hat{\lambda}}

\date{}

\begin{document}
\title{Reconstruction of symmetric Potts Models}

\author{Allan Sly}
\address{Allan Sly\hfill\break
Department of Statistics\\
UC Berkeley\\
Berkeley, CA 94720, USA.\\
Supported by NSF grants DMS-0528488 and DMS-0548249 and ONR grant N0014-07-1-05-06}
\email{sly@stat.berkeley.edu}
\urladdr{}

\begin{abstract}
The reconstruction problem on the tree has been studied in numerous contexts including statistical physics, information theory and computational biology.  However, rigorous reconstruction thresholds have only been established in a small number of models.  We prove the first exact reconstruction threshold in a non-binary model establishing the Kesten-Stigum bound for the 3-state Potts model on regular trees of large degree.  We further establish that the Kesten-Stigum bound is not tight for the $q$-state Potts model when $q \geq 5$.  Moreover, we determine asymptotics for the reconstruction thresholds.

\end{abstract}

\maketitle

\section{Introduction}

\subsection{Preliminaries}

We begin by giving a general description of broadcast (or Markov) models on
trees and the reconstruction problem.  The broadcast model on a tree
$T$ is a model in which information is sent from the root $\rho$
across the edges, which act as noisy channels, to the leaves of $T$.
For some given finite set of characters $\mathcal{C}$ a
configuration on $T$ is an element of $\mathcal{C}^T$, that is an
assignment of a character $\mathcal{C}$ to each vertex. We will denote the elements of $\mathcal{C}$ as $\{1,\ldots, q\}$ and  $q=|\mathcal{C}|$ as the number of characters. 
The broadcast model is a probability distribution on configurations
defined as follows.  Some $|\mathcal{C}|\times |\mathcal{C}|$
probability transition matrix $M$ is chosen as the noisy channel on
each edge.  The spin $\sigma_\rho$ is chosen from $\mathcal{C}$
according to some initial distribution and is then is propagated
along the edges of the tree according to the transition matrix $M$. That is if vertex $u$ is the parent of $v$ in the tree then the spin at $v$ is defined according to the probabilities
\[
P(\sigma_v = j|\sigma_u=i)=M_{i,j}.
\]
The focus of this paper is on the symmetric channel which are given by transition matrices of the form 
$$M_{i,j} =\begin{cases}1-p & \mathrm{if  }\ i=j, \\
\frac{p}{q-1} & \hbox{otherwise,} \end{cases}$$
where $0<p \leq 1$.  The state of the root is chosen according to the uniform distribution on $\mathcal{C}$.  

The symmetric channel corresponds to the $q$-state Potts model on the tree.  The Potts model weights configurations according to the Hamiltonian $H(\sigma)= \sum_{(u,v)\in E} 1_{\{\sigma_u=\sigma_v\}}$ which counts the number of edges in which the characters on each side are equal.  On a finite tree the probability distribution is given by
\[
P(\sigma) = \frac1{Z} \exp\left(\beta \sum_{(u,v)\in E} 1_{\{\sigma_u=\sigma_v\}} \right)
\]
where $Z$ is a normalising constant.  On an infinite tree more than one Gibbs measure
may exist, the symmetric channel corresponds to the free
Gibbs measure.  The two models coincide when $1-p=\frac{e^\beta}{e^\beta+q-1}$.  It will be convenient to parameterise the symmetric channel by its second largest eigenvalue by absolute value (that is either the second eigenvalue or the last eigenvalue, whichever is larger).  It is given by
\[
\lambda=\lambda(M)=1-\frac{pq}{q-1} = \frac{e^\beta-1}{e^\beta+q-1}
\]
and takes values in the interval $[-\frac1{q-1},1)$.  The special case of proper colourings corresponds to $\lambda=-\frac{1}{q-1}$.  In line with the terminology for the Potts model we will say the channel is ferromagnetic when $\lambda>0$ and anti-ferromagnetic when $\lambda<0$.

We will restrict our attention to $d$-ary trees, that is the
infinite rooted tree where every vertex has $d$ offspring.  Let
$\sigma(n)$ denote the spins at distance $n$ from the root and let
$\sigma^i(n)$ denote $\sigma(n)$ conditioned on $\sigma_\rho=i$.
\begin{definition}
We say that a model is \emph{reconstructible} on a tree $T$ if for
some $i,j\in \mathcal{C}$,
\[
\limsup_n d_{TV} (\sigma^i(n),\sigma^j(n)) > 0
\]
where $d_{TV}$ is the total variation distance. When the limsup is 0
we will say the model has \emph{non-reconstruction} on $T$.
\end{definition}
Non-reconstruction is equivalent to the mutual information between
$\sigma_\rho=\sigma(0)$ and $\sigma(n)$ going to 0 as $n$ goes to infinity and
also to $\{\sigma(n)\}_{n=1}^\infty$ having a trivial tail
sigma-field. In terms of Gibbs measures non-reconstruction is equivalent to the free measure being extremal, that is not a convex combination of two other Gibbs measures.  More equivalent formulations are given in
\cite{Mossel:04} Proposition 2.1.  In contrast consider the uniqueness property of a Gibbs measure.
\begin{definition}
We say that a model has \emph{uniqueness} on a tree $T$ if
\[
\limsup_n \quad\sup_{A,B} \ \ \ d_{TV}
\Big(P(\sigma_\rho=\cdot|\sigma(n)=A),P(\sigma_\rho=\cdot|\sigma(n)=B) \Big) > 0
\]
where the supremum is over all configurations $A,B$ on the vertices
at distance $n$ from the root.
\end{definition}
Reconstruction implies non-uniqueness and is a strictly stronger
condition.  Essentially uniqueness says that there is some
configuration on the leaves which provides information on the root
while reconstruction says that a typical configuration on the leaves
provides information on the root.

\subsection{Background}
For a given parameterized collection of models the key question in
studying reconstruction is finding which models have reconstruction,
which typically involves finding a threshold.  The reconstruction problem naturally arises in biology, information theory and statistical physics and involves the trade off between increasing numbers of leaves with increasingly noisy information as the distance from the root to the leaves increases.  In the case of the Potts model this is the question of for which $\lambda$ is there reconstruction for each choice of $q$ and $d$.  Proposition 12 of \cite{Mossel:01} implies that for each $q$ and $d$ there exist $\lambda^-<0<\lambda^+$ such that there is non-reconstruction when $\lambda\in (-\lambda^-,\lambda^+)$ and reconstruction when $\lambda\in [-\frac1{q-1},\lambda^-)\cup (\lambda^+,1)$.  The result does not say what happens when $\lambda\in\{\lambda^- , \lambda^+\}$.

The most general result on reconstruction is the Kesten-Stigum bound \cite{KesSti:66} which says that reconstruction holds when $\lambda^2 d> 1$ which in our parameterisation says that $\lambda^+\leq  d^{-1/2}$ and $\lambda^- \geq - d^{-1/2}$.  In fact when
$d\lambda^2 > 1$ it is possible to asymptotically
reconstruct the root from just knowing the number of times each
character appears on the leaves (census reconstruction) without
using the information on their positions on the leaves.

The simplest collection of models is the binary (2-state) symmetric channel which is defined on two characters and corresponds to the Ising model on the tree with no external field.  It was shown in \cite{BlRuZa95} and \cite{EvKePeSch:00} that this channel has reconstruction if and only if $d\lambda^2>1$, that is the Kesten-Stigum bound is sharp.
Before this paper exact reconstruction thresholds had only been calculated in the
binary symmetric channel and binary asymmetric channels with
sufficiently small asymmetry \cite{BoChMoRo:06} where the Kesten-Stigum is also sharp. Mossel
\cite{Mossel:01,Mossel:04} showed that the Kesten-Stigum bound is
not the bound for reconstruction in the binary-asymmetric model with
sufficiently large asymmetry or in the ferromagnetic Potts model with  $q\geq 18$.  For general Potts models \cite{MoPe:03} showed non-reonstruction when
\[
\frac{qd\lambda^2}{2+(q-2)\lambda}\leq 1
\]
and these bounds were improved in \cite{MSW:07}.  Several recent results deal with the special case of proper colourings which is now known to good accuracy.  By analysing a simple reconstruction algorithm reconstruction was shown to hold when $d \geq q[\log q + \log \log q + 1 + o(1)]$ see \cite{MoPe:03, Semerjian:08}.  The tightest bounds for non-reconstruction are $d\leq q[\log q + \log \log q + 1 - \log 2 + o(1)]$ established by \cite{sly:08}, the difference between the upper and lower bounds is just $q\log 2$.

Using techniques from statistical physics and including numerical simulations M{\'e}zard and Montanari \cite{MezMon:06} made a series of conjectures for the symmetric channels.

\begin{conjecture}[\cite{MezMon:06}]
The Kesten-Stigum bound is tight for the ferromagnetic symmetric channel when $q \leq 4$ and is not tight when $q \geq 5$.  In the anti-ferromagnetic model the Kesten-Stigum bound is tight when $q\leq 3$ and not tight when $q \geq 4$.
\end{conjecture}
As this conjecture was based on numerical evidence they qualified it by stating that it might not hold for large $d$.  This paper confirms much of the predicted picture.

\subsection{Main Results}

Our results confirm much of the picture predicted by Mezard and Montanari \cite{MezMon:06}.  We give a complete picture for large $d$ except in the case of $q=4$ which the proof will show is a critical case.  The $q=4$ case will be dealt with in a subsequent paper.

\begin{theorem}\label{t:q3nonrecon}
When $q=3$ there exists a $d_{\mathrm{min}}$ such that for $d\geq d_{\mathrm{min}}$ the Kesten-Stigum bound is sharp for both the ferromagnetic and antiferrmagnetic channels, that is $\lambda^+(d)= d^{-1/2}$ and $\lambda^-(d)=  -d^{-1/2}$.  Furthermore there is non-reconstruction at the  Kesten-Stigum bound, when $\lambda=\lambda^+$ or  $\lambda=\lambda^-$.
\end{theorem}

Conversely when $q\geq 5$ the Kesten-Stigum bound is never sharp.
\begin{theorem}\label{t:q5recon}
When $q\geq 5$ for every $d$ the Kesten-Stigum bound is not sharp, that is $\lambda^+ < d^{-1/2}$ and $\lambda^{-} >  -d^{-1/2}$.
\end{theorem}

\subsubsection{Asymptotic results}

When the Kesten-Stigum bound is not sharp we are not able to exactly compute the threshold, doing so involves finding a non-trivial fixed point of an equation of vector-valued distributions.  Nonetheless we are able to give precise asymptotics for the thresholds for fixed $q$ and $d$ goes to infintiy.  In light of the Kesten-Stigum bound it makes sense to consider $d^{1/2}\lambda^\pm$.  %When $q=4$ this is asymptotically $\pm 1$, that is the Kesten-Stigum bound.  
When $q\geq 5$ the limit is strictly different from 1.

\begin{theorem}\label{t:reconasym}
%When $q=4$, 
%\begin{align*}
%\lim_{d\to\infty} d^{1/2}\lambda^+ &=1\\
%\lim_{d\to\infty} d^{1/2}\lambda^- &=-1
%\end{align*}
% while 
When $q\geq 5$,
 \begin{align*}
\lim_{d\to\infty} d^{1/2}\lambda^+ &=C_q\\
\lim_{d\to\infty} d^{1/2}\lambda^- &=-C_q
\end{align*}
where $C_q$ is a constant strictly less than 1.
\end{theorem}

Of course when $q=3$ we have that $d^{1/2}\lambda^\pm=\pm 1$ for large $d$.

%
%\subsection{Robust Reconstruction}

%

%\begin{theorem}\label{t:q3nonreconRobust}
%When $q=3$ these is not robust reconstruction at the Kesten-Stigum bound.
%\end{theorem}

%\begin{theorem}\label{t:q5reconRobust}
%When $q=5$ there is robust reconstruction at the Kesten-Stigum bound.
%\end{theorem}

\subsection{Applications}

The broadcast model is a natural model for the evolution of
characters of DNA.  In phylogenetic reconstruction the goal is to
reconstruct the ancestry tree of a collection of species given their
genetic data. Establishing a 
 conjecture of Mike Steel it was shown  that 
the number of samples required for phylogenetic reconstruction
undergoes a phase transition at the reconstruction threshold for the
binary symmetric channel \cite{Mossel:04b,DaMoRo:06}.

The reconstruction threshold on trees is believed to play a critical
role in the dynamic phase transitions in certain glassy systems given by random constraint satisfaction problems such as random K-SAT and the anti-ferromagnetic Potts model on random graphs. We will briefly describe the broad picture conjectured by
physicists about such systems \cite{KMRSZ:07,ZdKr:07},
generally without rigorous proof, and why understanding the
reconstruction threshold for colourings plays an important role in
such systems.  The theory relates to the structure and connectivity of the set of configurations which support most of the measure of the distribution, with the topology given by the hamming distance on the space of configurations.

At ``high temperatures'' or low densities of constraints the Gibbs measure places all but an exponentially small fraction of its weight in a single  ``connected cluster''.  As the temperature decreases there is a threshold called the ``dynamical replica symmetry breaking threshold'' at which the set supporting most of the measure splits into exponentially many smaller clusters.  The clusters are each well separated from each other and contain an exponentially small amount of the measure but together contain all but an exponentially small amount of the measure.  This threshold is believed to correspond to the reconstruction threshold on the corresponding tree model.  In a recent result \cite{AchlioptasCoOg:08} rigorously proved that for random colourings on Erd\H{o}s-R\'enyi random graphs with average connectivity $d$ when $(1+o(1)) q\log q \leq d \leq (2- o(1)) q\log q$ the space of solutions indeed breaks into exponentially many small clusters.  The lower bound corresponds to the reconstruction threshold for colourings on the tree \cite{sly:08}.

Another threshold, the condensation threshold, is believed to occur at even lower temperatures.  At this point clusters exist with a positive fraction of the measure, these masses are believed to be jointly given by a Poisson-Dirichlet distribution.  When the Kesten-Stigum bound is tight these thresholds coincide and there is no phase where the clusters all have a small proportion of the mass.

The reconstruction threshold is also believed to play an important role
in the efficiency of the Glauber dynamics on  trees and random graphs.  In \cite{BCMY:05} it was shown that the mixing time for the the Glauber dynamics on trees is $n^{1+\Theta(1)}$ when the model has reconstruction and slower than at higher temperature when the mixing time is $O(n\log n)$.  In the case of the Ising model this is tight, the mixing time is $O(n\log n)$ when $d \lambda^2< 1$.  

Local MCMC algorithms are conjectured to be
efficient up to the reconstruction threshold for sampling random colourings on random graphs but experience an
exponential slowdown beyond it \cite{KMRSZ:07}.  This is to be
expected since a local MCMC algorithm can not move between clusters
each of which has exponentially small probability.  Rigorous proofs
of rapid mixing of MCMC algorithms, such as the Glauber dynamics,
fall a long way behind.  For colourings of random regular graphs, results of
\cite{DFHV:04} imply rapid mixing when $q\geq 1.49 d$, well
below the reconstruction threshold and even the uniqueness
threshold. Even less is known for Erd\H{o}s-R\'enyi random graphs as
almost all MCMC results are given in terms of the maximum degree
which in this case grows with $n$. Polynomial time mixing of the
Glauber dynamics has been shown \cite{MoSly:07} for a constant
number of colours in terms of $d$, the average connectivity.

\subsection{Proof Sketch}

The proof analyses a quantity denoted by $x_n$.  One interpretation of $x_n$ is that if we guess the value of $\sigma_\rho$ according to its postier distribution given $\sigma(n)$ then $x_n$ is the probability of being correct minus $\frac1q$, which is the chance of being correct by simply guessing randomly.  More formally if $Z$ is a $\mathcal{C}$-valued random variable with distribution given by $P(Z=i\mid \sigma(n))=P(\sigma_\rho=i \mid \sigma(n))$ then $x_n = P(Z=\sigma_\rho) -\frac1q$.  Our analysis is similar to the expansion of \cite{BCRM:06} but with more precise estimates derived by establishing concentration results.  Such expansions go back to \cite{CCST:86} in the context of spin-glasses.

We show that $x_n$ is always positive and that non-reconstruction is equivalent to
\[
\lim_{n\to \infty} x_n=0.
\] 
In general finding the recnostruction threshold requires understanding recursive equations of vector-valued distributions c.f. \cite{MezMon:06}.  However, when $x_n$, the amount of information about the between the leaves and the root, is small and the equations become close to linear.  Using Taylor series expansions and concentration estimates establishes that for small $x_n$
\begin{equation}\label{e:introExpansion}
x_{n+1} = d \lambda^2  x_n + (1+o(1)) \frac{d(d-1)}{2}  \frac{q(q-4)}{q-1}  \lambda^4 x_n^2 .
\end{equation}
A key role is played by the sign of $q-4$.  When $q\geq 5$ it is positive and this allows us to show that if $d\lambda^2$ is sufficienty close to 1 then $x_n$ does not converge to 0 and hence there is reconstruction beyond the Kesten-Stigum bound.  

However, when $q=3$ the second order term is negative.  Suppose we could establish that $x_n$ is eventually small when $d\lambda^2 \leq 1$.  Then equation \eqref{e:introExpansion} implies that $x_n$ converges to 0 which establishes non-reconstruction.  Unfortunately for small $d$ we are not able to show that $x_n$ becomes sufficiently small to apply this argument.

When $d$ is large the interactions between spins become very weak but there are many of them.  Using the Central Limit Theorem we approximate this collection of small independent interactions to show that
\[
x_{n+1} \approx g_q(d \lambda^2 x_n),
\]
for some increasing function $g_q$.  When $q=3$ for all $0<s<1$ the function satisfies $g_3(s)<s$.  Using this estimate for large enough $d$ it is established that $x_n$ become arbitrarily small.  Combining this with equation \eqref{e:introExpansion} proves non-reconstruction for large enough $d$.  When $q=4$ for all $0<s<1$ the function also satisfies $g_4(s)<s$ while when $q\geq 5$ the equation $g_5(s)=s$ has nonzero solutions.  The function $g_q(s)$ determines the limiting value of $x_n$, a consequence of which is Theorem \ref{t:reconasym}.

\section{Proofs}

We introduce the notation we use in the proofs.  We denote the
colours by $\mathcal{C}=\{1,\ldots,q\}$ and let $T$ be the
$d$-ary tree rooted at $\rho$. Let $u_1,\ldots,u_d$ be the
children of $\rho$ and for a vertex $v\in T$ let $T_v$ denote the subtree of descendants
of $v$ (including $v$).  Throughout the paper we will use the convention that $i$ will denote an element of $\mathcal{C}$ and $j$ will be an element of $\{1,\ldots,d\}$ corresponding to a child of $\rho$.  Let $\sigma$ denote a random configuration given by the symmetric channel with
transition matrix given by
$$M_{i,j} =\begin{cases}1-p & \mathrm{if  }\ i=j, \\
\frac{p}{q-1} & \hbox{otherwise,} \end{cases}$$
where $0<p \leq 1$.  Rather than looking at the unconditioned configurations $\sigma$ we will work mainly with configurations where the spin at the root is conditioned; we let $\sigma^i$ denote a random configuration according to the the symmetric channel conditioned on $\sigma_\rho^i=i$.    Let $\lambda$ denote the second eigenvalue of $M$ which is given by
\begin{equation}\label{e:lambdaRelation}
\lambda=\lambda(M)=1-\frac{pq}{q-1}.
\end{equation}
In light of the Kesten-Stigum bound we will always assume that $d \lambda^2 \leq 1$.

Let $S(n)$ denote the vertices on level $n$, $\{v\in T: d(v,\rho)=n\}$, let $\sigma(n):=\sigma_{S(n)}$ denote the spins on $S(n)$ and let $\sigma_j(n)$ denote the spins in $S(n)\cap T_{u_j}$.  For a configuration $A$ on $S(n)$ define the posterior function $f_n$ as
\[
f_n(i,A) = P(\sigma_\rho=i | \sigma(n) = A).
\]
By the recursive nature of the tree for a configuration $A$ on $S(n+1)\cap T_{u_j}$ we also have (with a slight abuse of notation) that
\[
f_n(i,A) = P(\sigma_{u_j}=i | \sigma_j(n+1) = A).
\]
Now define  $X_i(n)=X_i$ by
\[
X_i(n)=f_n(i,\sigma(n)).
\]
These random variables are a deterministic function of the random
configuration $\sigma(n)$ of the leaves which gives the posterior
probability that the root is in state $i$.  Recall that a collection of random variables are exchangeable if their distribution is invariant under permutations.  By symmetry the $X_i$
are exchangable. Now we define two random variables
\[
X^+ = X^{+}(n)=f_n(1,\sigma^{1}(n))
\]
and
\[
X^- = X^{-}(n)=f_n(2,\sigma^{1}(n)).
\]
We will
establish non-reconstruction (respectively reconstruction) by showing that $X^+$ and $X^-$ both
converge (resp. do not converge) to $\frac1q$ in probability as $n$ goes to infinity.     By symmetry we have
\[
f_n(i_2,\sigma^{i_1}(n))\stackrel{d}{=}\begin{cases} X^+ & i_1=i_2, \\
X^- & \hbox{otherwise,} \end{cases}
\]
and the set $\{f_n(i,\sigma^{1}(n)):2\leq i \leq q\}$ is
exchangeable.  Moreover they are conditionally exchangeable given $f_n(1,\sigma^{1}(n))$.  

Now define
\[
Y_{ij}= Y_{ij}(n) = f_n(i,\sigma^{1}_j(n+1)).
\]
This is none other than the posterior probability that $\sigma_{u_j}=i$ given the random configuration $\sigma^{1}_j(n+1)$ on the spins in $S(n)\cap T_{u_j}$.  Conditional on the spin at the root the spins in the subtrees $T_{u_j}$ are conditionally independent for $j=1,\ldots,d.$  Taking advantage of this and the symmetries of the model the following proposition is immediate.

\begin{proposition}\label{p:symmetries}
The $Y_{ij}$ satisfy the following properties:
\begin{itemize}
\item The random vectors $Y_j=\left(Y_{1j},\ldots,
Y_{qj} \right)$ are independent for $j=1,\ldots,d.$

\item Conditional on $\sigma_{u_j}$ the random variable $Y_{\sigma_{u_j} j}$ is equal in distribution to $X^+(n)$ while for $i\neq \sigma_{u_j}$ the random variables $Y_{ij}$ are equal in distribution to  $X^-(n)$.

\item Further given $\sigma_{u_j}$ and $Y_{\sigma_{u_j} j}$ the random variables $\{Y_{ij}\}_{i\neq \sigma_j }$ are conditionally exchangeable.  

\end{itemize}
\end{proposition}

The key method of this paper will be to analyze the relation between the distributions $X^+(n)$ and $X^+(n+1)$ using the recursive structure of the tree.  Suppose $A$ is a configuration  on $S(n+1)$ and let $A_j$ be its restriction to $T_{u_j}\cap S(n+1)$. The following standard relation follows from the Markov random field property
\begin{align}\label{e:recursive}
f_{n+1}(1,A)&= \frac{\prod_{j=1}^d  \left( M_{11} f_n(1,A_j) + \sum_{l \neq 1} M_{1l}
f_n(l,A_j) \right)}{\sum_{i=1}^q \prod_{j=1}^d \left( M_{ii} f_n(i,A_j) + \sum_{l \neq i}M_{il}
f_n(l,A_j) \right) }\nonumber\\
&=\frac{\prod_{j=1}^d  \left(M_{12} + (M_{11}-M_{12})f_n(1,A_j)  \right)}{\sum_{i=1}^q \prod_{j=1}^d \left( M_{12} + (M_{11}-M_{12})  f_n(i,A_j)\right) }\nonumber\\
&=\frac{\prod_{j=1}^d  \left(1 + \lambda q (f_n(1,A_j)-\frac1q)  \right)}{\sum_{i=1}^q \prod_{j=1}^d \left( 1 + \lambda q ( f_n(i,A_j) -\frac1q ) \right) }
\end{align}
where the second equality follows from the fact that $\sum_{i=1}^q f_n(i,A_j)=1$ and the symmetry of $M$ and the final equality follows from equation \eqref{e:lambdaRelation} since
\[
M_{12}+\frac1q\left(M_{11}-M_{12}\right) = M_{12}+\frac1q\left(1-(q-1)M_{12}-M_{12}\right)=\frac1q
\]
and
\[
M_{11}-M_{12} = 1-qM_{12} =\lambda.
\]
Conditioning the root to be 1 and letting $A=\sigma^1(n+1)$ we have that
\begin{equation}\label{e:mainZexpression}
X^{+}(n+1)= \frac{Z_1}{\sum_{i=1}^k Z_i}
\end{equation}
where
\begin{equation}\label{e:zDefn}
Z_i= Z_i(n) = \prod_{j=1}^d \left(1+ \lambda q (Y_{ij}(n)-\frac1q)\right).
\end{equation}
Equation \eqref{e:mainZexpression} will be our major tool for recursive analysing the reconstruction problem.

\subsection{Basic Identities}
Denote
$$x_n=E(X^{+}(n)-\frac1q) = E f_n(1,\sigma^{1}(n))-\frac1q$$ and $$z_n=E(X^{+}(n)-\frac1q)^2 = E (f_n(1,\sigma^{1}(n))-\frac1q)^2.$$ As discussed in the introduction the main proof relies on analysing recursions of $x_n$.  This is based on the approach of \cite{BCRM:06} used in the binary asymmetric channel but with a more refined analysis, in particular establishing concentration of the random variables $X_i$.  The following lemma, which can be viewed as the analogue of Lemma 1 of
\cite{BCRM:06}, allows us to relate the first and second moments of
$X^+$.

\begin{lemma}\label{l:changeOfMeasure}
The following relations hold:
\[
x_n+\frac1q = EX^+ = E
\sum_{i=1}^q (X_i(n))^2=E(X^+(n))^2 + (q-1)E(X^-(n))^2,
\]
and
\[
x_n =  E \sum_{i=1}^q (X_i(n)-\frac1q)^2= E(X^+(n)-\frac1q)^2 + (q-1)E(X^-(n)-\frac1q)^2 \geq
 z_n.
\]
\end{lemma}
\begin{proof}
From the definition of conditional probabilities and of $f_n$ and the fact that
$P(\sigma_\rho=1)=\frac1q$ we have that
\begin{align*}
EX^+(n)
&=Ef_n(1,\sigma^{1}(n)) \\
&= \sum_A f_n(1,A)
P(\sigma(n) = A|\sigma_\rho=1)\\
&= \sum_A \frac{P(\sigma(n)=A,\sigma_\rho=1)}{P(\sigma_\rho=1)} f_n(1,A)\\
&= q\sum_A P(\sigma(n)=A) f_n(1,A)^2\\
&= q E(X_1(n))^2\\
&= E \sum_{i=1}^q (X_i(n))^2
\end{align*}
and
\[
E \sum_{i=1}^k (X_i(n)-\frac1q)^2  = E \sum_{i=1}^q (X_i(n))^2 -
\frac2q E \sum_{i=1}^q X_i(n) + \frac1{q}= EX^+-\frac1q.
\]
Conditional on $\sigma_\rho$ we have that $X_{\sigma_\rho}(n)$ is distributed as $X^+(n)$ and for $i\neq \sigma_\rho$ we have that $X_i(n)$ is distributed as $X^-(n)$.  It follows that
\[
E \sum_{i=1}^q (X_i(n))^2=E(X^+(n))^2 + (q-1)E(X^-(n))^2
\]
and
\[
E \sum_{i=1}^q (X_i(n)-\frac1q)^2=E(X^+(n)-\frac1q)^2 + (q-1)E(X^-(n)-\frac1q)^2
\]
which completes the result.
\end{proof}

Define $\hat{\sigma}_\rho(n)$ to be the maximum likelihood estimator of $\sigma_\rho$ given $\sigma(n)$ which is given by
\[
\hat{\sigma}_\rho(n) := \hbox{argmax}_i X_i(n)
\]
where in the case that multiple states maximize the likelihood the, estimator chooses randomly between these states.  This estimator maximizes the probability of correctly reconstructing the root.  Define the probability of correct reconstruction as
\[
p_n := P\left(\sigma_\rho = \hat{\sigma}_\rho(n)\right) = E\max_{1\leq i \leq q} X_i(n)
\]
This represents the probability of correctly reconstructing the spin at the root using the maximum likelihood estimator which maximizes the probability of correctly determining the root.  Since $\sigma(n)$ is a Markov process $p_n$ is clearly decreasing.

%Since $\sigma(n+1)$ and $\sigma_\rho$ are conditionally independent given $\sigma(n)$ for information theoretic reasons clearly $p_n$ is decreasing.

\begin{lemma}\label{l:reconProb}
We have that
\[
x_n \leq p_n -\frac1q \leq x_n^{1/2}
\]
\end{lemma}

\begin{proof}
The inequality $x_n +\frac1q \leq p_n$ was shown in \cite{MezMon:06}  by noting that the algorithm that chooses $\sigma_\rho$ randomly according to probabilities $X_i$ is correct with probability $x_n +\frac1q$.  By the Cauchy-Schwartz inequality and Lemma \ref{l:changeOfMeasure}
\begin{align*}
p_n &= E\max_i X_i \leq \frac1q + E \max_i \left| X_i -\frac1q\right| \leq \frac1q + \left( E \max_i \left(X_i -\frac1q\right)^2 \right)^\frac12\\
&\leq \frac1q + \left( E \sum_{i=1}^q \left(X_i -\frac1q\right)^2 \right)^\frac12 = \frac1q +x_n^{1/2}
\end{align*}
as required.
\end{proof}

The following corollary of Lemmas \ref{l:changeOfMeasure} and \ref{l:reconProb} justifies our focus on $x_n$.
\begin{corollary}\label{c:changeOfMeasure}
We have that $x_n \geq 0$ and the condition
\[
\lim_n x_n = 0.
\]
is equivalent to non-reconstruction.
\end{corollary}
\begin{proof}
Lemma \ref{l:changeOfMeasure} implies that $x_n \geq z_n \geq 0$.  By Lemma \ref{l:changeOfMeasure},  $x_n$ converging to 0 is equivalent to 
\[
\sum_{i=1}^k E\left(X_i(n)-\frac1q\right)^2\rightarrow 0
\]
which is equivalent to the posteriors converging to the stationary distribution which is in turn equivalent to reconstruction \cite{Mossel:04}.
%If $x_n$
%converges to $0$ then
%\[
%\sum_{i=1}^k E\left(X_i(n)-\frac1q\right)^2\rightarrow 0
%\]
%which implies non-reconstruction.  If $x_n$ does not converge to 0 then since $p_n$ is decreasing by Lemma \ref{l:reconProb}  it is bounded away from $\frac1q$ which establishes reconstruction.
\end{proof}

Using the identities from Lemma \ref{l:changeOfMeasure} we calculate the means and covariances of the $Y_{ij}$.  

\begin{lemma}\label{l:identities}
For each $1\leq j \leq q$ the following hold:
\begin{equation}\label{e:identityA}
E (Y_{1j} -\frac1q)= \lambda x_n, \quad E (Y_{1j}-\frac1q)^2 =  \lambda z_n + \frac1q (1-\lambda) x_n.
\end{equation}
For $i\neq 1$ we have that
\begin{equation}\label{e:identityB}
E (Y_{ij} -\frac1q) = -\frac{\lambda x_n}{q-1}, \quad  E (Y_{ij}-\frac1q)^2 =\frac1q(1+\frac{\lambda}{q-1})x_n - \frac{\lambda}{q-1}z_n,
\end{equation}
and
\begin{equation}\label{e:identityC}
E(Y_{1j} -\frac1q)(Y_{ij}-\frac1q)= - \frac{\lambda}{q-1}z_n - \frac{1-\lambda}{q(q-1)} x_n.
\end{equation}
When $1<i_1 < i_1 \leq q$,
\begin{equation}\label{e:identityD}
E(Y_{i_1j}-\frac1q)(Y_{i_2j}-\frac1q)=\frac1{(q-1)(q-2)} \left[ 2\lambda z_n - \frac1q (q-2+2\lambda) x_n\right].
\end{equation}

\end{lemma}

\begin{proof}
By Proposition \ref{p:symmetries} if $\sigma^1_{u_j}=1$ then $Y_{1j}$ is distributed according to $X^+(n)$ otherwise it is distributed according to $X^-(n)$.  By equation \eqref{e:lambdaRelation} we have that 
\[
P(\sigma_{u_j}^1=1)=\frac{1+\lambda(q-1)}{q}
\]
Noting that $\sum_{i=1}^q Y_{ij} = 1$ it follows that $EX^+(n) + (q-1)EX^-(n)=1$ and so $E (X^{-}(n)-\frac1q)=-\frac{x_n}{q-1}$.
It follows that
\begin{align*}
E (Y_{1j} -\frac1q)&=  P(\sigma_{u_j}^1=1)E(X^+(n) -\frac1q) + (1-P(\sigma_{u_j}^1=1))E (X^{-}(n)-\frac1q)\\
&= \frac{1+\lambda(q-1)}{q}  x_n + \left(1-\frac{1+\lambda(q-1)}{q} \right)\frac{-x_n}{q-1}\\
&= \lambda x_n.
\end{align*}
Using Lemma \ref{l:changeOfMeasure} and Proposition \ref{p:symmetries} we have that,
\begin{align}\label{e:expressionForYSquared}
E (Y_{1j}-\frac1q)^2 &=  P(\sigma_{u_j}^1=1) E(X^+(n)-\frac1q)^2 +  (1-P(\sigma_{u_j}^1=1)) E (X^{-}(n) -\frac1q)^2\nonumber\\
&= \frac{1+\lambda(q-1)}{q}  z_n +\left(1-\frac{1+\lambda(q-1)}{q} \right) \frac1{q-1} \left[E(X^+(n)-\frac1q) -E(X^+-\frac1q)^2\right ]  \nonumber\\
&= \lambda z_n + \frac1q (1-\lambda) x_n
\end{align}
which establishes equation \eqref{e:identityA}.   Now since $\sum_{l=1}^q Y_{lj} =1$ and since by Proposition \ref{p:symmetries} we have that $Y_{2j},\ldots,Y_{qj}$ are exchangeable, for $i\neq 1$ we have that
\begin{align*}
E (Y_{ij} -\frac1q) &= \frac1{q-1} \sum_{l=2}^q E (Y_{lj} -\frac1q)\\
&=  -\frac1{q-1} E (Y_{1j} -\frac1q)\\
&= -\frac{\lambda x_n}{q-1}.
\end{align*}
Again using Lemma \ref{l:changeOfMeasure} and the exchangeability of $Y_{2j},\ldots,Y_{qj}$ we have that,
\begin{align*}
E (Y_{ij}-\frac1q)^2 &= \frac1{q-1}\left[- E (Y_{1j}-\frac1q)^2 +  \sum_{l=1}^q E (Y_{lj}-\frac1q)^2\right]\\
&=  \frac1{q-1}\left[- (\lambda z_n + \frac1q (1-\lambda) x_n) +  x_n \right]\\
&= \frac1q(1+\frac{\lambda}{q-1})x_n - \frac{\lambda}{q-1}z_n.
\end{align*}
By the fact that $\sum_{l=2}^q (Y_{lj}-\frac1q)=-(Y_{1j} -\frac1q)$,
\begin{align*}
E(Y_{1j} -\frac1q)(Y_{ij}-\frac1q)&=\frac1{q-1} \sum_{l=2}^q E(Y_{1j} -\frac1q)(Y_{lj}-\frac1q)\\
&= - \frac1{q-1} E(Y_{1j} -\frac1q)^2\\
&= - \frac{\lambda}{q-1}z_n - \frac{1-\lambda}{q(q-1)} x_n
\end{align*}
where the third equality follows from equation \eqref{e:expressionForYSquared}.  Finally 
\begin{align*}
E(Y_{i_1j}-\frac1q)(Y_{i_2j}-\frac1q)&=  \frac1{(q-1)(q-2)} E \left[(Y_{1j} -\frac1q)^2 - \sum_{l=2}^q (Y_{lj} -\frac1q)^2 \right]\\
&=  \frac1{(q-1)(q-2)} \bigg[ (\lambda z_n + \frac1q (1-\lambda) x_n)\\
&\quad\quad - (q-1)( \frac1q(1+\frac{\lambda}{q-1})x_n - \frac{\lambda}{q-1}z_n)) \bigg]\\
&=\frac1{(q-1)(q-2)} \left[ 2\lambda z_n - \frac1q (q-2+2\lambda) x_n\right]
\end{align*}

%Checksum
%\begin{align}
%A z_n + \frac1q (1-A) x_n + (q-1)(   \frac1q(1+\frac{A}{q-1})x_n - \frac{A}{q-1}z_n)) \\
% +  2(q-1) (- \frac{A}{q-1}z_n - \frac{1-A}{q(q-1)} x_n)  +  (2A z_n - \frac1q (q-2+2A) x_n)\\
% = z_n ( A - A - 2A + 2A)\\
% + \frac1q x_n(1-A + q-1 + A -2 + 2A - q + 2 - 2A)\\
% =0
%\end{align}

%\begin{align*}
%\hbox{Cov}(Y_{i_1j},Y_{i_2j})&=\\
%\end{align*}
\end{proof}

\subsection{Taylor Series Bounds}

In the following lemma we calculate expected values of monomials of the $Z_i$ by expanding them using Taylor series approximations. 

\begin{lemma}\label{l:taylor}
For each positive integer $k$, there exists a $C=C(q,k)$ not depending $\lambda$ or $d$ such that for each $0 \leq  k_1 ,\ldots, k_q, \leq k$,
\[
E\prod_{i=1}^q Z_i^{k_i} \leq C
\]
and
\[
\left| E\prod_{i=1}^q Z_i^{k_i}  - 1 - d \left(E\prod_{i=1}^q \left( 1 + \lambda q (Y_{i1}-\frac1q)\right)^{k_i} -1\right) \right|  \leq C x_n^2
\]
and
\begin{align*}
&\Bigg|  E\prod_{i=1}^q Z_i^{k_i} -1 - d \left(E\prod_{i=1}^q \left( 1 + \lambda q (Y_{i1}-\frac1q)\right)^{k_i} -  1\right) \\ &\quad \quad -  \frac{ d(d-1)}{2} \left( E\prod_{i=1}^q \left( 1 + \lambda q (Y_{i1}-\frac1q)\right)^{k_i}  -  1 \right)^2 \Bigg| \leq C x_n^3.
\end{align*}

\end{lemma}

\begin{proof}

Recall that
\[
Z_i= Z_i(n) = \prod_{j=1}^d \left(1+ \lambda q (Y_{ij}(n)-\frac1q)\right)
\]
so each $Z_i$ is a product of independent and identically distributed terms and that
\[
E\prod_{i=1}^q Z_i^{k_i} = \left( E \prod_{i=1}^q \left(1+ \lambda q (Y_{1j}(n)-\frac1q)\right)^{k_i} \right)^d.
\]
As such we begin with a simple bound on $(1+y)^d$ using Taylor series.  Suppose that $d |y| \leq C'$ for some constant $C'>0$.  Then we have that,
\begin{align}\label{e:taylorA}
\left |  (1+y)^d  -  \sum_{i=0}^\ell  {d \choose i} y^i \right| &\leq \sum_{i=\ell+1}^d  {d \choose i} |y|^i \nonumber\\ 
&\leq \sum_{i=\ell+1}^\infty  \frac{d^i}{i!} |y|^i  \nonumber\\ 
&=  e^{d|y|} - \sum_{i=0}^\ell \frac{(d|y|)^i}{i!}   \nonumber\\ 
&\leq e^{C'} |d y|^{\ell+1}
\end{align}
where the third inequality follows by Taylor's Theorem since $\max_{x \leq C'} \frac{d^{\ell+1}}{dx^{\ell+1}} e^x = e^{C'}$.

Suppose that $s_1,\ldots,s_q$ are nonnegative integers.  If for some $\ell$,  $s_\ell \geq 2$ then since by definition $0 \leq Y_{ij}\leq 1$, by Lemma \ref{l:changeOfMeasure},
\begin{equation}\label{e:taylorB}
 \left| E \prod_{i=1}^q (Y_{i1}-\frac1q)^{s_i} \right | \leq E (Y_{\ell 1}-\frac1q)^{2} \leq x_n.
\end{equation}
If for distinct integers $\ell,\ell'$, $s_\ell=s_{\ell'}=1$ then again by by Lemma \ref{l:changeOfMeasure},
\begin{align}\label{e:taylorC}
\left| E \prod_{i=1}^q (Y_{i1}-\frac1q)^{s_i} \right |& \leq E \left| (Y_{\ell 1}-\frac1q)(Y_{\ell' 1}-\frac1q)\right  | \nonumber\\
& \leq   E \left[ (Y_{\ell 1}-\frac1q)^2 + (Y_{\ell' 1}-\frac1q)^2 \right] \leq x_n .
\end{align}
Finally if $s_\ell=1$ and $s_i=0$ for all $i\neq \ell$ then by Lemma  \ref{l:identities},
\begin{equation}\label{e:taylorD}
\left| E \prod_{i=1}^q (Y_{i1}-\frac1q)^{s_i} \right | = \left| E Y_{\ell 1}-\frac1q \right | \leq |\lambda| x_n.
\end{equation}
Then applying equations \eqref{e:taylorB}, \eqref{e:taylorC} and \eqref{e:taylorD},
\begin{align*}
&\left| E\prod_{i=1}^q \left( 1 + \lambda q (Y_{i1}-\frac1q)\right)^{k_i} - 1 \right | \\
&= \left | \sum_{(s_1,\ldots,s_q)} E \prod_{i=1}^q {k_i \choose s_i} \lambda^{s_i} q^{s_i} \left(Y_{i1}-\frac1q\right)^{s_i}  -1 \right|\\
&= \left |E\sum_{i=1}^q k_i \lambda q(Y_{i1}-\frac1q) +  \sum_{(s_1,\ldots,s_q), \sum s_i \geq 2} E \prod_{i=1}^q {k_i \choose s_i} \lambda^{s_i} q^{s_i} \left(Y_{i1}-\frac1q\right)^{s_i}   \right|\\
& \leq C' \lambda^2 x_n
\end{align*}
where the sum runs over all $q$-tuples of nonegative integers $(s_1,\ldots,s_q)$ with $s_i\leq k_i$ for all $i$ and the constant $C'$ depends only on $q$ and $k_1,\ldots, k_q$.  The final inequality in the last equation follows from equations \eqref{e:taylorB}, \eqref{e:taylorC} and \eqref{e:taylorD} since every term is bounded by $C'' \lambda^2 x_n$ where $C''$ depends only on $q$ and $k$.  Since $0\leq x_n \leq 1$ and $\lambda^2d \leq 1$ applying equation \eqref{e:taylorA} with $$y=E\prod_{i=1}^q \left( 1 + \lambda q (Y_{i1}-\frac1q)\right)^{k_i} - 1$$
completes the result.
\end{proof}

\subsection{Main Expansion}

In order to evaluate the expected value of $E X^+(n+1)$ using equation \eqref{e:mainZexpression} we expand it out using the identity
\begin{equation}\label{e:basicExpansionIdentity}
\frac{a}{s+r}=\frac{a}s - \frac{a r}{s^2} + \frac{r^2}{s^2}\frac{a}{s+r}.
\end{equation}
With this expansion and $a=Z_1$, $s=q$ and $r= (\sum_{i=1}^q Z_i)-q$
clearly,
\begin{align}\label{e:mainEquation}
x_{n+1} &=   E \frac{Z_1}{\sum_{i=1}^q Z_i} -\frac1q\nonumber\\
&=  \frac{Z_1}{q} -E
\frac{Z_1\left((\sum_{i=1}^q Z_i)- q \right)}{q^2}   + E \frac{Z_1}{\sum_{i=1}^q
Z_i}\frac{\left((\sum_{i=1}^q Z_i)-q\right)^2}{q^2} -\frac1q.
\end{align}
We estimate the expected value of each of the terms in the preceding equation.  First 
\begin{align}\label{e:ExpansionA}
EZ_1 &= 1 + d   \lambda q E (Y_{11}-\frac1q) +  \frac{ d(d-1)}{2} \left( \lambda q E (Y_{11}-\frac1q) \right)^2 + R_1\nonumber\\
&= 1 + d \lambda^2 q x_n + \frac{d(d-1)}{2} \lambda^4 q^2 x_n^2 + R_1
\end{align}
where by Lemma \ref{l:taylor} the error term  satisfies $|R_1| \leq C_1 x_n^3$ where $C_1$ does not depend on $\lambda, d$ or $x_n$.  Next applying Lemma \ref{l:taylor} and Lemma \ref{l:identities} and cancelling terms
\begin{align}\label{e:ExpansionB}
EZ_1 \left(\sum_{i=1}^q Z_i- q \right) 
&= EZ_1^2  + \sum_{i=2}^q E Z_1 Z_i - q E Z_1\nonumber\\
&= \frac{d(d-1)}{2}\lambda^4 q^2 \Bigg[ \left((3-\lambda) x_n +  \lambda q z_n \right)^2 \nonumber\\
&+ \frac1{q-1} \big( (q-3+\lambda)x_n   -  \lambda q z_n  \big)^2 -q   x_n^2 \Bigg] + R_2
\end{align}
where by Lemma \ref{l:taylor}  $|R_2| \leq C_2 x_n^3$ and $C_2$ does not depend on $\lambda, d$ or $x_n$.  Finally again using Lemma \ref{l:taylor} and Lemma \ref{l:identities},
\begin{align}\label{e:ExpansionC}
E\left((\sum_{i=1}^q Z_i)- q \right)^2
&= EZ_1^2  + \sum_{i=2}^q EZ_i^2 + 2 \sum_{i=2}^q E Z_1 Z_i + \sum_{i_1=2}^q \sum_{i_2=i_1+1}^q E Z_{i_1}Z_{i_2} \nonumber \\
&- 2qE Z_1 - 2q\sum_{i=2}^q E Z_i + q^2  \nonumber\\
&=  \frac{d(d-1)}{2} \lambda^4 q^2 \Bigg[     \left( (3-\lambda) x_n +  \lambda q z_n  \right)^2 \nonumber \\
&  + \frac3{q-1} \left( (q-3+\lambda) x_n -\lambda q z_n\right)^2  - 2q x_n^2 -\frac{2q x_n^2}{q-1}  \nonumber\\
& +\frac1{(q-1)(q-2)} \left( (3q-6-2\lambda) x_n +2\lambda q z_n \right)^2  \Bigg] + R_3
\end{align}
 where by Lemma \ref{l:taylor} $|R_3| \leq C_3 x_n^3$ and $C_3$ does not depend on $\lambda, d$ or $x_n$.  By Lemma  \ref{l:changeOfMeasure} we have that $0\leq z_n \leq x_n$ and since $|\lambda|\leq 1$ the expressions in equations \eqref{e:ExpansionB} and \eqref{e:ExpansionC} are both bounded by $C \frac{d(d-1)}{2} \lambda^4 x_n^2$ where $C$ depends only on $q$.  Now using the fact that $0\leq \frac{Z_1}{\sum Z_i}\leq 1$ and substituting equations \eqref{e:ExpansionA}, \eqref{e:ExpansionB} and \eqref{e:ExpansionC} into equation \eqref{e:mainEquation} we have that
 \begin{equation}\label{e:crudeExpansion}
\left| x_{n+1} -d \lambda^2  x_n\right| \leq C_q \lambda^4 \frac{d(d-1)}{2}x_n^2 \leq C_q x_n^2
\end{equation}
where $C_q$ depends only on $q$ since $\lambda^2 d \leq 1$.  In order to complete the proof we will need a more precise bound.  To motivate the rest of the proof suppose that we could establish the following condition: 
\begin{condition}\label{co:assumptions}
Suppose the following holds:
\begin{itemize}

\item That $z_n= (\frac1q +o(1))x_n$,
\item That $ \frac{Z_1}{\sum_{i=1}^q
Z_i}$ is sufficiently concentrated around $\frac1q$ so that 
\[
E \frac{Z_1}{\sum_{i=1}^q
Z_i}\frac{\left((\sum_{i=1}^q Z_i)-q\right)^2}{q^2} = \Big(\frac1q+o(1)\Big)E\frac{\left((\sum_{i=1}^q Z_i)-q\right)^2}{q^2}
\]
 \end{itemize}
 \end{condition}
If we established Condition \ref{co:assumptions}  then by substituting equations \eqref{e:ExpansionA}, \eqref{e:ExpansionB} and \eqref{e:ExpansionC} into equation \eqref{e:mainEquation} we would have that 
\begin{equation}\label{e:mainExpansion}
x_{n+1}  =    d \lambda^2  x_n + (1+o(1)) \frac{q(q-4)}{q-1}   \frac{d(d-1)}{2}   \lambda^4 x_n^2.
\end{equation}
Proving Condition \ref{co:assumptions} is one of the main technical challenges in this paper.

\subsection{Concentration Lemmas}\label{ss:Concentration}

In this subsection we establish a number of lemmas in order to establish the Condition \ref{co:assumptions}.  The following lemma follows immediately from equation \eqref{e:crudeExpansion}.

\begin{lemma}\label{l:smallChange}
For any $\epsilon>0$, there exists a constant $\delta=\delta(q,\epsilon)$ such that for all $n$, if $x_n<\delta$ then
\[
\left | x_{n+1} - d \lambda^2 x_n \right| \leq \epsilon x_n.
\] 
\end{lemma}

The following lemma ensures that the decrease from $x_n$ to $x_{n+1}$ is never too large.

\begin{lemma}\label{l:constantFactorDecrease}
For any $\kappa>0$ there exists a constant $\gamma=\gamma(q, \kappa,d)>0$ such that for all $n$ when $\kappa <|\lambda|$,
\[
x_{n+1} \geq \gamma x_n.
\] 
\end{lemma}

\begin{proof}
For a configuration $A$ on $T_{u_1} \cap S(n+1)$ define
\[
f^*_{n+1}(i,A) =  P(\sigma_{\rho}=i | \sigma_1(n+1) = A) ~;
\]
that is the probability the root is in state 1 given the configuration on the leaves in $T_{u_1} \cap S(n+1)$.  Now 
\begin{align*}
f^*_{n+1}(i,A) &=  \frac{\left( e^\beta f_n(1,A) + \sum_{l \neq 1}
f_n(l,A) \right)}{\sum_{i=1}^q  \left( e^\beta f_n(i,A) + \sum_{l \neq i}
f_n(l,A) \right) }\\&=  \frac{  \left(1 +  \lambda q (f_n(1,A)-\frac1q) \right)}{q},
\end{align*}
and so
\[
E f^*_{n+1}(i,\sigma_1^1(n)) = \frac1q + \lambda^2 x_n
\]
The estimator that chooses a state with probability $f^*_{n+1}(i,\sigma_1(n))$ correctly reconstructs the root with probability $\frac1q + \lambda^2 x_n$.  Since this probability must be less than the MLE it follows that
\[
\lambda^2 x_n + \frac1q \leq p_{n+1} \leq  x_{n+1}^{1/2} + \frac1q.
\]
and so $x_{n+1} \geq \lambda^4 x_n^2 \geq \kappa ^4 x_n^2$ for an value of $x_n$.  Now when $x_n< \delta$ by Lemma \ref{l:smallChange} it follows that
\[
x_{n+1} \geq (d\lambda^2-\epsilon) x_n.
\]
Combining these results completes the proof.
\end{proof}
\subsubsection{Concentration}

We will establish some concentration results which will be required in order to make the approximation
\[
\frac{Z_1}{\sum_{i=1}^q Z_i} \approx \frac1q.
\]
The first lemma establishes a technical uniqueness result where the set of vertices which can be conditioned is limited to a set of $k$ vertices.

\begin{lemma}\label{l:uniformUnq}
For any $\epsilon>0$ and positive integer $k$ there exists $\Lambda=\Lambda(q,d,\epsilon,k)$ not depending on $\lambda$ such that for any collection of vertices $v_1,\ldots,v_k \in S(\Lambda),$
\[
\sup_{i, i_1,\ldots,i_k \in \mathcal{C}} \left | P\left(  \sigma_\rho = i | \sigma_{v_j}=i_j, 1\leq j \leq k   \right)  - \frac1q \right| < \epsilon.
\]
\end{lemma}

\begin{proof}
This lemma simply says that fixing the spins at $k$ distant vertices a long way from the root has only a small effect on the root.  We note that
\[
M^s_{i_1,i_2} =   \begin{cases} \frac1q + (1-\frac1q)\lambda^s & i_1=i_2, \\
 \frac1q  -\frac1q \lambda^s& \hbox{otherwise,} \end{cases}
\]
and so since $\lambda^2 d \leq 1$,
\[
\frac1q-d^{-s/2} \leq M^s_{i_1,i_2}  \leq \frac1q+d^{-s/2}.
\]
Let $\gamma$ be an integer sufficiently large such that
\[
\left(\frac{ \frac1q-d^{-\gamma/2}}{ \frac1q-d^{-\gamma/2}} \right)^k <1+\epsilon.
\]
Fix an integer $\Lambda$ such that $\Lambda > k\gamma$.  Now choose any $v_1,\ldots,v_k \in S(\Lambda)$ with $d(v_i,\rho)=\Lambda$.  For $0 \leq \ell \leq \Lambda$ define $a_\ell$ to be the number of vertices distance $\ell$ from the root with a decedent in the set $\{v_1,\ldots,v_k\}$, that is $a_\ell = \#\{ v\in S(\ell): |T_v \cap\{v_1,\ldots,v_k\}|>0 \}$.  Then $a_0=1, a_\Lambda=k$ and the $a_\ell$ are increasing and integer valued.  Therefore there must be some $\ell$ such that $a_\ell = a_{\ell+\gamma}$.  Let $\overline{w}_1,\ldots,\overline{w}_{a_\ell}$ denote the vertices in the set $\{ v\in S(\ell) : |T_v \cap\{v_1,\ldots,v_k\}|>0 \}$ and  $w_1,\ldots,w_{a_\ell}$ denote the vertices in the set $\{ v\in S(\ell +\gamma): |T_v \cap\{v_1,\ldots,v_k\}|>0 \}$ such that $w_j$ is the descendent of $\overline{w}_j$.  By the Markov random field property the $\sigma_{w_j}$ are conditionally independent given the $\sigma_{\overline{w}_j}$.  The distribution of $\sigma_{w_j}$ given  $\sigma_{\overline{w}_j}$ is
\[
P(\sigma_{w_j} = i_2| \sigma_{\overline{w}_j}=i_1) =  M^\gamma_{i_1,i_2}.
\]
By Bayes Rule and the Markov random field property we have that for any $i,i',i_1,\ldots,i_{a_\ell} \in \mathcal{C}$,
\begin{align*}
& \frac{P\left(\sigma_\rho=i \mid \sigma_{w_j}= i_j, 1\leq j \leq a_{\ell} \right) }{P(\sigma_\rho=i' | \sigma_{w_j}= i_j, 1\leq j \leq a_{\ell} )} \\
&=  \frac{P( \sigma_{w_j}= i_j, 1\leq j \leq a_{\ell}  |\sigma_\rho=i )}{P(\sigma_{w_j}= i_j', 1\leq j \leq a_{\ell} |\sigma_\rho=i')} \\
&= \frac{\sum_{h_1,\ldots,h_{a_\ell} \in \mathcal{C} } P(\forall j \ \sigma_{w_j}= i_j  |  \forall j \ \sigma_{\overline{w}_j}= h_j )  P(\forall j \ \sigma_{\overline{w}_j}= h_j | \sigma_\rho=i )}{\sum_{h_1,\ldots,h_{a_\ell} \in \mathcal{C}} P( \forall j \ \sigma_{w_j}= i_j |  \forall j \  \sigma_{\overline{w}_j}= h_j)  P(\forall j 
\ \sigma_{\overline{w}_j}= h_j  | \sigma_\rho=i' )}\\
&= \frac{\sum_{h_1,\ldots,h_{a_\ell} \in \mathcal{C}} P(\sigma_{\overline{w}_j}= h_j, 1\leq j \leq a_{\ell}| \sigma_\rho=i ) \prod_{j=1}^{a_\ell} M_{h_j,i_j}^\gamma }{ \sum_{h_1,\ldots,h_{a_\ell} \in \mathcal{C}} P(\sigma_{\overline{w}_j}= h_j, 1\leq j \leq a_{\ell}| \sigma_\rho=i' ) \prod_{j=1}^{a_\ell} M_{h_j,i_j}^\gamma}\\
&\leq \frac{\sum_{h_1,\ldots,h_{a_\ell} \in \mathcal{C}} P(\sigma_{\overline{w}_j}= h_j, 1\leq j \leq a_{\ell}| \sigma_\rho=i ) \left(\frac1q+d^{-\gamma/2} \right)^{a_\ell} }{ \sum_{h_1,\ldots,h_{a_\ell} \in \mathcal{C}} P(\sigma_{\overline{w}_j}= h_j, 1\leq j \leq a_{\ell}| \sigma_\rho=i' )\left(\frac1q  -  d^{-\gamma/2} \right)^{a_\ell} }\\
&\leq \frac{\left(\frac1q+d^{-\gamma/2} \right)^{a_\ell} }{  \left(\frac1q-d^{-\gamma/2} \right)^{a_\ell}   }\\
&\leq 1+\epsilon.
\end{align*}
so it follows that
\[
P(\sigma_\rho=i | \sigma_{w_j}= i_j, 1\leq j \leq a_{\ell} ) \leq \frac1q(1+\epsilon)
\]
and
\[
P(\sigma_\rho=i | \sigma_{w_j}= i_j, 1\leq j \leq a_{\ell} ) \geq \frac1q\frac1{1+\epsilon} \geq \frac1q(1-\epsilon).
\]
By the Markov random field property since $\sigma_\rho$ is conditionally independent of the collection $\sigma_{v_1},\ldots,\sigma_{v_k}$ given the spins $\sigma_{w_1},\ldots,\sigma_{w_{a_\ell}}$ it follows that,
\begin{align*}
&\sup_{i, i_1,\ldots,i_k \in \mathcal{C}} \left | P\left(  \sigma_\rho = i | \sigma_{v_j}=i_j, 1\leq j \leq k   \right)  - \frac1q \right|\\  
& \quad \leq \sup_{i, i_1,\ldots,i_{a_\ell} \in \mathcal{C}} \left | P\left(  \sigma_\rho = i | \sigma_{w_j}=i_j, 1\leq j \leq a_{\ell}   \right)  - \frac1q \right| < \epsilon
\end{align*}
which completes the result.

\end{proof}

The next lemma establishes concentration of the posterior distributions when $x_n$ is small.

\begin{lemma}\label{l:concentrationWeak}
For any $\epsilon,\alpha, \kappa >0$ there exists $C=C(q,d,\epsilon,\alpha,\kappa)$ and $N=N(q,d,\epsilon,\alpha,\kappa)$ such that for any $\lambda$ with $\kappa <|\lambda|\leq d^{-1/2}$ and for $n>N$,
\[
P\left(\left|\frac{Z_1}{\sum_{i=1}^q Z_i}- \frac1q  \right |> \epsilon\right ) \leq C x_n^{\alpha}.
\]

\end{lemma}

\begin{proof}

The conclusion is trivially true is both $C$ and $x_n$ are large so we will suppose that $x_n$ is small.  Fix $k$ an integer such that $k>\alpha$. Choose $\Lambda$ large enough so that the conclusion of Lemma \ref{l:uniformUnq} holds with bound $\epsilon/2$ and set $N=\Lambda$.  Let $v_1,\ldots, v_{|S(\Lambda)|}$ denote the vertices in $S(\Lambda)$.  Let $\sigma^1_v(n+1)$ denote the spins of the vertices in $T_v\cap S(n+1)$ and define
\[
W(i,v) = f_{n-\Lambda}(i,\sigma^1_v(n+1))
\]
which is the conditional probability that $\sigma_v$ is in state $i$ given the boundary condition $\sigma^1_v(n)$.  Conditional on $\sigma^1(\Lambda)$, the spins of $S(\Lambda)$, the  $W(i,v)$ are distributed as
\[
W(i,v) \sim   \begin{cases} X^+(n+1-\Lambda) & \sigma^1_v= i, \\
X^-(n+1-\Lambda) & \sigma^1_v\neq i. \end{cases}
\]
Conditional on $\sigma(\Lambda)$ the vectors $(W(1,v),\ldots,W(q,v))$ are conditionally independent for different $v\in S(\Lambda)$.   Using the recursion of equation \eqref{e:recursive} a posterior probability of a vertex can be written as a function of the posterior probabilities of its children so there exists a function $g_\lambda(\mathcal{W})$ such that,
\[
\frac{Z_1}{\sum_{i=1}^q Z_i} =  f_n(1,\sigma^{1}(n+1)) = g_\lambda(\mathcal{W})
\]
where $\mathcal{W}$ denotes the vector $$\mathcal{W}=\left(W(1,v_1),\ldots,W(1,v_{|S(\Lambda)|}),W(2,v_1),\ldots,W(q,v_{|S(\Lambda)|})\right).$$  When $x_n$ is small we expect most of the $W(i,v)$ to be close to $\frac1q$.  If all the entries in $\mathcal{W}$ are identically $\frac1q$ then $g_\lambda(\mathcal{W})=\frac1q$.  It follows by Lemma \ref{l:uniformUnq} that if there are at most $k$ vertices $v\in S(\Lambda)$ such that for some $1\leq i \leq q$, $W(i,v) \neq \frac1q$ then 
\[
\left|g_\lambda(\mathcal{W})-\frac1q \right| < \epsilon/2.
\]
Observe that $g_\lambda$ is a continuous function of each of the elements of the vector  $\mathcal{W}$ and of $\lambda$. It follows that there exists a $\delta>0$ such that if $\mathcal{W}$ satisfies
\[
\#\left\{ v\in S(\Lambda) : \max_{1\leq i \leq q} \left| W(i,v) - \frac{1}{q} \right| > \delta \right\} \leq k
\]
then $$\left|g_\lambda(\mathcal{W})-\frac1q \right| < \epsilon.$$
As the random variables $\max_{1\leq i \leq q} \left| W(i,v) - \frac1q \right|$ are independent  since they are conditionally independent given $\sigma(\Lambda)$ and by the symmetry of the model they do not in fact depend on the spins in $S(\Lambda)$.  By Chebyshev's inequality and Lemma \ref{l:changeOfMeasure} we have that
\begin{align*}
&P\left( \max_{1\leq i \leq q} \left| W(i,v) - \frac1q \right| > \delta \right)\\
\quad &\leq P\left(\left|X^+(n+1-\Lambda)-\frac1q\right| > \delta \right) + (q-1)P\left(\left|X^-(n+1-\Lambda)-\frac1q\right| > \delta \right)\\
\quad &\leq \delta^{-2}\left[ E(X^+(n+1-\Lambda)-\frac1q)^2 + (q-1)E(X^-(n+1-\Lambda)-\frac1q)^2 \right]\\
\quad &= \frac{q}{\delta^2}x_{n+1-\Lambda}.
\end{align*}
As noted above we may suppose that $x_n$ is very small so these events are rare.  In particular we have that
\begin{align*}
P\left(\left|\frac{Z_1}{\sum_{i=1}^q Z_i} - \frac1q  \right |> \epsilon\right ) &\leq P\left(\#\left \{\max_{1\leq i \leq q} \left| W(i,w_j) - \frac1q \right| > \delta \right\} > k \right)\\
&\leq P\left( \mathrm{Binom}\left( |S(\Lambda)|,   \frac{q}{\delta^2}x_{n-\Lambda} \right)> k \right)\\
& \leq C' x_{n+1-\Lambda}^{\alpha}\\
&\leq C x_{n}^{\alpha}
\end{align*}
where the third inequality holds for large enough $C'$ and  the final inequality follows by Lemma \ref{l:constantFactorDecrease} which completes the proof.  Only in this final inequality do we use the assumption that $\kappa <|\lambda|$.

\end{proof}

To establish the necessary concentration results we will make use of Bennet's inequality which is stated below (see e.g. \cite{Pollard:84}(Appendix B, Lemma 4).
\begin{lemma}\label{l:priorConcentration}
For independent mean 0 random variables $W_1, \ldots ,W_n$ satisfying
$W_i \leq M, b_n^2 = \sum\limits_{i=1}^n E(W_i^2).$ Then for any $\eta \ge 0$,
\begin{equation}\label{e:priorConcentration}
P(\sum_{i=1}^n W_i \ge \eta) \leq \exp \left(-\frac{b_n^2}{M^2} \theta \left(\frac{\eta M}{b_n^2}\right)\right)
\end{equation}
where $\theta(x) = (1+x)\log(1+x)-x$.
\end{lemma}
The following concentration result holds uniformly provided $\lambda$ is small enough.  It is necessary in taking limits for large $d$.

\begin{lemma}\label{l:concentration}
For any $0<\epsilon<1$ and $\alpha>1$ there exists $C=C(q,\epsilon,\alpha)$  and $N=N(q,\epsilon,\alpha)$ depending only on $q$, $\alpha$ and $\epsilon$ such that whenever $|\lambda |q \leq \frac12$ and
\[
|\lambda|q+\lambda^2 q^2 \leq \frac{\max\{-\log(1-\epsilon),\log(1+\epsilon)\}}{4\alpha}
\]
then for $1\leq i \leq q$ and $n>N$,
\[
P\left(\left|Z_i(n) - 1  \right |> \epsilon\right ) \leq C x_n^{\alpha}.
\]
\end{lemma}

\begin{proof}
Observe that the hypothesis only holds when $|\lambda|$ is small, that is the interactions are weak enough.  Let 
\[
M=\frac{\max\{-\log(1-\epsilon),\log(1+\epsilon)\}}{4 \alpha}.
\]
By taking $C$ large enough we can assume that $$x_n <\frac{q^2}{2} \min \{-\log(1-\epsilon),\log(1+\epsilon)\},$$ since otherwise the conclusion is trivial.

Since $1-2y \leq \frac1{1+y} \leq 1$ when $0\leq y \leq \frac12$ and $1-2y \geq \frac1{1+y} \geq 1$ when $-\frac12 \leq y \leq 0$ by integrating it follows that when $|y| \leq \frac12$, 
\begin{equation}\label{e:simpleLog}
y-y^2 \leq \log(1+y) \leq y.
\end{equation}
Taking $y= \lambda q(Y_{ij}-\frac1q)$ then,
\[
-M\leq -|\lambda q|-\lambda^2q^2 \leq \lambda q(Y_{ij}-\frac1q) -  \lambda^2q^2(Y_{ij}-\frac1q)^2 \leq \log(1+ \lambda q(Y_{ij}-\frac1q) ),
\]
and
\[
\log(1+ \lambda q(Y_{ij}-\frac1q) ) \leq  \lambda q(Y_{ij}-\frac1q) \leq |\lambda q|\leq M.
\]
Let $$W_j=\lambda q(Y_{1j}-\frac1q) -  \lambda^2 q^2 (Y_{1j}-\frac1q)^2$$ and so by Lemma \ref{l:identities},  $$EW_j = \lambda^3 q x_n - \lambda^2 q^3 z_n\leq |\lambda|^3 q x_n$$ and $-(W_j-EW_j) \leq M + |\lambda|^3 q \leq 2M$.  Also $EW_j = \lambda^3 q x_n - \lambda^3 q^2 z_n\geq -|\lambda|^3 q^2 x_n$ so $ d EW_j \geq -q^2 x_n$. Since by definition, $0\leq Y_{ij} \leq 1$, our assumption that $|\lambda| q<\frac12$ implies that $|\lambda q(Y_{1j}-\frac1q)|< \frac12$.  From the inequality $(a+b)^2 \leq 2a^2 +2b^2$ and Lemma \ref{l:identities} it follows that
\[
E(W_j-EW_j)^2\leq E W_j^2 \leq 2 E\left(\lambda q(Y_{1j}-\frac1q)\right)^2 + 2E\left(\lambda q(Y_{1j}-\frac1q)\right)^4 \leq 4\lambda^2 q^2 x_n.
\]
and so if $B=\sum_{j=1}^d E(W_j-EW_j)^2$ then $B\leq 4 d \lambda^2 q^2 x_n \leq 4q^2 x_n$ since $d\lambda^2\leq 1$.  Now
\begin{align}\label{e:concentration}
P\left(Z_1 \leq 1-\epsilon \right) &= P\left(   \sum_{j=1}^d \log\left(1+ \lambda q(Y_{1j}-\frac1q)\right) \leq \log(1-\epsilon)  \right)\nonumber\\
&\leq P\left( \sum_{j=1}^d W_j \leq \log(1-\epsilon)  \right)  \nonumber\\
&\leq P\left( \sum_{j=1}^d -(W_j -E W_j ) \geq -\log(1-\epsilon)  -q^2 x_n \right)  \nonumber\\
&\leq P\left( \sum_{j=1}^d -(W_j -E W_j ) \geq -\frac12 \log(1-\epsilon)  \right)  \nonumber\\
&\leq   \exp \left(-\frac{B}{4M^2} \theta \left(\frac{(-\frac12 \log(1-\epsilon)) 2M}{B}\right)\right).
\end{align}
where the first inequality follows from the equation \eqref{e:simpleLog},  the second from the fact that  $ d EW_j \geq -q^2 x_n$, the third from our assumption that $x_n <\frac{q^2}{2} \max \{-\log(1-\epsilon),\log(1+\epsilon)\}$ and the final inequality by applying Lemma \ref{l:priorConcentration}

Since $\frac1x\theta(x)$ is increasing in $x$ the right hand side of equation \eqref{e:concentration} is increasing in $B$ and hence substituting $B \leq 4q^2 x_n$ gives,
\begin{align}\label{e:concentration2}
P\left(Z_1 \leq 1-\epsilon \right) & \leq   \exp \left(-\frac{4q^2 x_n}{4M^2} \theta \left(\frac{-\log(1-\epsilon)M}{4q^2 x_n}\right)\right)  \nonumber\\
& \leq   \exp \left[-\frac{-\log(1-\epsilon)}{4M}\left ( \log \left(\frac{-\log(1-\epsilon) M}{4q^2 x_n}\right)   -  1 \right) \right]  \nonumber\\
&\leq   \exp \left[\frac{\log(1-\epsilon)}{4M}\left ( \log \left(\frac{-\log(1-\epsilon) M}{4q^2}\right)   -  1 \right) \right]  x_n^{-\frac{\log(1-\epsilon)}{4M}}\nonumber\\
&\leq C x_n^{\alpha}.
\end{align}
where the second inequality uses the fact that $\theta(x) < x(\log(x)-1)$.  With essentially the same argument we have $P(Z_1 \geq 1+\epsilon) <  C x_n^{\alpha}$.  Furthermore the result holds similarly for the other $Z_i$ as well which completes the result.
\end{proof}
%
%  Let 
%\[
%M=\frac{\max\{-\log(1-\epsilon),\log(1+\epsilon)\}}{4\alpha}.
%\]
%We will assume the condition that
%\begin{itemize}
%\item $|A|q \leq \frac12$ and $|A|q+A^2 q^2 \leq M$.
%\end{itemize}

%Since $A^2d \leq 1$ this condition holds if $d$ is large enough so that $d\geq 4q^2$ and $d^{-1/2q}+d^{-1}q^2\leq M$.  For the finite number of values of $d$ for which the condition does not always hold, the result follows immediately from Lemma \ref{l:concentrationWeak} so we will herein assume that the condition does hold.  It is from the use of Lemma \ref{l:concentrationWeak} that we require that $n>N$.
Combining the results of this section the following corollary gives us the concentration result we need.

\begin{corollary}\label{c:concentration}
For any $0<\epsilon<1$ and $\alpha>1$ there exists $C=C(q,\epsilon,\alpha)$  and $N=N(q,\epsilon,\alpha)$ depending only on $q$, $\alpha$ and $\epsilon$ such that for $1\leq i \leq q$ and $n>N$,
\begin{equation}\label{e:corConcentration}
P\left(\left|\frac{Z_1}{\sum_{i=1}^q Z_i} - \frac1q  \right |> \epsilon\right ) \leq C x_n^{\alpha}.
\end{equation}
\end{corollary}

\begin{proof}
In light of Lemmas \ref{l:concentrationWeak}  and Lemma \ref{l:concentration} we split the result into two cases, when $|\lambda|$ is big and  small.
Let $\epsilon'(q)>0$ be small enough so that if for all $i$, $|Z_i - 1| < \epsilon'$ then
\[
\left|\frac{Z_1}{\sum_{i=1}^q Z_i} - \frac1q  \right | < \epsilon,
\]
and let
\[
M=\frac{\max\{-\log(1-\epsilon'),\log(1+\epsilon')\}}{4\alpha}.
\]
For each fixed $d$ define $$\mathcal{K}_d=\{ \lambda:  |\lambda |q < \frac12, |\lambda|q+\lambda^2 q^2 < M\},$$ an open set which includes 0.  Let $\mathcal{J}_d= [-d^{-1/2},d^{1/2}] \setminus \mathcal{K}_d$.

By Lemma \ref{l:concentration} equation \eqref{e:corConcentration} holds with a bound $C'=C'(q,\epsilon,\alpha)$ not depending on $\lambda$ or $d$, provided $\lambda \in \mathcal{K}_d$.  For each fixed $d$ Lemma \ref{l:concentrationWeak} implies that equation \eqref{e:corConcentration} holds with a bound $C''_{d}=C''_{d}(q,\epsilon,\alpha)$ not depending on $\lambda$, provided $\lambda \in \mathcal{J}_d$.  Since $\lambda^2 d \leq 1$, for large enough $d$ so that  $d\geq 4q^2$ and $d^{-1/2}q+d^{-1}q^2\leq M$ the set $\mathcal{J}_d$ is empty.  It follows that equation \eqref{e:corConcentration} holds with a bound $$C=\max\left\{C', \max_{d': \mathcal{J}_{d'} \neq \phi } C''_{d'} \right \}$$ that is independent of $\lambda$ and $d$.
\end{proof}

\subsection{Bound on $z_n-\frac1q x_n$}\label{ss:znBound}

In this section we bound the term $z_n-\frac1q x_n$ when $x_n$ is small.
\begin{lemma}\label{l:boundzn}

For any $\epsilon,\kappa>0$ there exists a $\delta=\delta(q,\kappa,d)$ and $k=k(q,\kappa,d)$ such that if $x_n<\delta$ and $| \lambda |\geq \kappa$ then
\[
\left| \frac{z_{n+k}}{ x_{n+k}} - \frac1q \right| \leq \epsilon.
\]
\end{lemma}

\begin{proof}

Using the identity \eqref{e:basicExpansionIdentity} we have 
\begin{align}\label{e:expandzn}
z_{n+1} &= E \frac{\left(Z_1 - \frac1q\sum_{i=1}^q Z_i \right)^2}{\left( \sum_{i=1}^q Z_i \right)^2} \nonumber\\
&=   E  \frac1{q^2} \left(Z_1 - \frac1q \sum_{i=1}^q Z_i \right)^2 - \frac1{q^4} \left(Z_1 - \frac1q \sum_{i=1}^q Z_i \right)^2 \left(  \left( \sum_{i=1}^q Z_i \right)^2  -q^2  \right) \nonumber \\
&\quad +  \frac1{q^4}  \frac{\left(Z_1 - \frac1q \sum_{i=1}^q Z_i \right)^2}{  \left( \sum_{i=1}^q Z_i \right)^2  } \left(  \left( \sum_{i=1}^q Z_i \right)^2  -q^2  \right)^2.  
\end{align}
Expanding and using Lemma \ref{l:taylor} and Lemma \ref{l:identities} we get that
\[
\left| E  \frac1{q^2} \left(Z_1 - \frac1q \sum_{i=1}^q Z_i \right)^2 - d \lambda^2 \left((1-\lambda) \frac1q x_n + \lambda z_n \right) \right| \leq C_q x_n^2.
\]
Similarly
\[
\left| E   \frac1{q^4} \left(Z_1 - \frac1q \sum_{i=1}^q Z_i \right)^2 \left(  \left( \sum_{i=1}^q Z_i \right)^2  -q^2  \right) \right| \leq C_q x_n^2
\]
and 
\[
E \left(  \left( \sum_{i=1}^q Z_i \right)^2  -q^2  \right)^2  \leq C_q x_n^2
\]
Substituting these bounds into equation \eqref{e:expandzn} and noting that 
\[
\left| \frac{\left(Z_1 - \frac1q \sum_{i=1}^q Z_i \right)^2}{  \left( \sum_{i=1}^q Z_i \right)^2  } \right| \leq 1
\]
so we have that
\[
\left| z_{n+1} - d \lambda^2  \left((1-\lambda) \frac1q x_n + \lambda z_n \right) \right|  \leq C_q' x_n^2.
\]
Dividing by $x_{n+1}$ we get
\[
\left| \frac{z_{n+1}}{x_{n+1}} - \frac{d\lambda^2 x_n}{x_{n+1}} \left (  (1-\lambda)\frac1q + \lambda\frac{z_n}{x_n} \right ) \right|  \leq C_q' \frac{x_n^2}{x_{n+1}}.
\]
By Lemma \ref{l:constantFactorDecrease} we have that $\frac{x_n}{x_{n+1}}\leq \gamma^{-1}$ and by equation \eqref{e:crudeExpansion} $|\frac{d \lambda^2 x_n}{x_{n+1}} - 1| \leq C'''_q \frac{x_n^2}{x_{n+1}}$.   It follows that
\begin{equation}\label{e:boundRationzn}
\left | \frac{z_{n+1}}{x_{n+1}} - \left( (1-\lambda)\frac1q + \lambda\frac{z_n}{x_n} \right) \right| \leq C_q'' x_{n+1}.
\end{equation}
Iterating this equation we get that
\begin{align}\label{e:iterativeZnBound}
&\left|\frac{z_{n+k}}{x_{n+k}} - (1-\lambda^k)\frac1q + \lambda^k\frac{z_n}{x_n}  \right|\nonumber \\
&\quad \leq \sum_{\ell=1}^k  \left| (1-\lambda^{k-\ell})\frac1q + \lambda^{k-\ell}\frac{z_{n+\ell}}{x_{n+\ell}} - (1-\lambda^{k-\ell+1})\frac1q - \lambda^{k-\ell+1}\frac{z_{n+\ell-1}}{x_{n+\ell-1}}  \right| \nonumber \\
&\quad \leq \sum_{\ell=1}^k  |\lambda|^{k-\ell} \left | \frac{z_{n+\ell}}{x_{n+\ell}} - \left( (1-\lambda)\frac1q + \lambda\frac{z_{n+\ell-1}}{x_{n+\ell-1}} \right) \right|  \nonumber \\
&\quad \leq C_q'' \sum_{\ell=1}^k  |\lambda|^{k-\ell} x_{n+\ell-1}.
\end{align}
Iteratively applying Lemma \ref{l:smallChange} implies that if $\delta>0$ is small enough and $x_n< \delta$ then for $0\leq \ell \leq k$, $x_{n+\ell} \leq 2\delta$.  Since $0\leq z_n \leq x_n$ it follows from equation \eqref{e:iterativeZnBound} that
\[
\left| \frac{z_{n+k}}{ x_{n+k}} - \frac1q \right| \leq \lambda^{k} + 2\delta C_q'' \sum_{\ell=1}^k  \lambda^{k-\ell}
\]
By taking $k$ sufficiently large and $\delta$ sufficiently small we complete the result.

\end{proof}

\begin{corollary}\label{c:boundzn}

For any $\epsilon,\kappa>0$ there exists a $\delta=\delta(q,\kappa,d)$ and $k=k(q,\kappa,d)$ such that if $x_n<\delta$, $n>k$ and $|\lambda|\geq \kappa$ then
\[
\left| \frac{z_{n}}{ x_{n}} - \frac1q \right| \leq \epsilon.
\]

\end{corollary}

\begin{proof}

By Lemma \ref{l:constantFactorDecrease} if $x_n < \delta$ then $x_{n-k} < \gamma^{-k} x_{n}$ and so the result follows by Lemma \ref{l:boundzn}.

\end{proof}

\section{Reconstruction for $q\geq 5$}

The lemmas proved in Subsections \ref{ss:Concentration} and \ref{ss:znBound} establish Condition \ref{co:assumptions}.  We now use these results to establish the change from $x_{n}$ to $x_{n+1}$ when $x_n$ is small. 

\begin{lemma}\label{l:xnIncreasing}
There exists a $\delta=\delta(q)>0$ and $N=N(q)$ such that if $x_n\leq \delta$ and $n>N$ then
\[
x_{n+1} \geq d \lambda^2  x_n + \frac12 \frac{d(d-1)}{2}  \frac{q(q-4)}{q-1}  \lambda^4 x_n^2 .
\]
\end{lemma}

\begin{proof}

Let $\epsilon>0$.  Then
\begin{align}\label{e:interchangeIntegrationusingConcen}
&\left| E \frac{Z_1}{\sum_{i=1}^q
Z_i}\frac{\left((\sum_{i=1}^q Z_i)-q\right)^2}{q^2} - E \frac1q \frac{\left((\sum_{i=1}^q Z_i)-q\right)^2}{q^2} \right|  \nonumber\\
 \leq & \epsilon E \frac1q \frac{\left((\sum_{i=1}^q Z_i)-q\right)^2}{q^2} + E I\left(\left| \frac{Z_1}{\sum_{i=1}^q Z_i} - \frac1q \right| >\epsilon \right)  \frac{\left((\sum_{i=1}^q Z_i)-q\right)^2}{q^2} \nonumber \\
\leq & \epsilon E \frac1q \frac{\left((\sum_{i=1}^q Z_i)-q\right)^2}{q^2} +  P\left(\left| \frac{Z_1}{\sum_{i=1}^q Z_i} -\frac1q \right| >\epsilon \right)^{\frac12}  \left( E\left( \frac{\left((\sum_{i=1}^q Z_i)-q\right)^2}{q^2}\right)^2\right)^{1/2} \nonumber\\
\leq& \epsilon E \frac1q \frac{\left((\sum_{i=1}^q Z_i)-q\right)^2}{q^2} +  C' x_n^3  \left( E\left( \frac{\left((\sum_{i=1}^q Z_i)-q\right)^2}{q^2}\right)^2\right)^{1/2}\nonumber\\
\leq &  \epsilon E \frac1q \frac{\left((\sum_{i=1}^q Z_i)-q\right)^2}{q^2} + C x_n^3 
\end{align}
where the second inequality comes from the Cauchy-Schwartz inequality and the third follows by Corollary \ref{c:concentration} provided that $n$ is sufficiently large while the fourth inequality follows by  Lemma \ref{l:taylor}.

Now by substituting equations \eqref{e:ExpansionA}, \eqref{e:ExpansionB} and \eqref{e:ExpansionC} we have that

\begin{align}\label{e:expandApproximate}
&E \frac{Z_1}{q} -E \frac{Z_1\left((\sum_{i=1}^q Z_i)- q \right)}{q^2}+ E \frac1q \frac{\left((\sum_{i=1}^q Z_i)-q\right)^2}{q^2} \nonumber\\
&\quad= \frac1q + d \lambda^2  x_n + \frac{d(d-1)}{2}\lambda^4  \Bigg[ \frac{2q(q-2)}{q-1} x_n^2  \nonumber\\
& \quad \quad- \frac{q-2}{q-1} \big( (q-3+\lambda)x_n   -  \lambda q z_n  \big)^2  -\frac{q-3}{q(q-1)}  \left( (q-3+\lambda) x_n -\lambda q z_n\right)^2  \nonumber\\ 
& \quad \quad +\frac1{q(q-1)(q-2)} \left( (3q-6-2\lambda) x_n +2\lambda q z_n \right)^2 \Bigg] + R \nonumber\\
& \quad\geq \frac1q + d \lambda^2  x_n + \frac{d(d-1)}{2}  \frac{q(q-4)}{q-1}  \lambda^4 x_n^2 \nonumber\\
& \quad \quad - C'  \frac{d(d-1)}{2}\lambda^5 \left|\frac{z_n}{x_n} - \frac1q \right|  x_n^2 - R
\end{align}
%
%\begin{align}\label{e:expandApproximate}
%&E \frac{Z_1}{q} -E \frac{Z_1\left((\sum_{i=1}^q Z_i)- q \right)}{q^2}+ E \frac1q \frac{\left((\sum_{i=1}^q Z_i)-q\right)^2}{q^2} \nonumber\\
%&= \frac1q + d A^2  x_n + \frac{d(d-1)}{2}A^4 q^2 x_n^2 \nonumber\\
%&-\frac1{q^2} \frac{d(d-1)}{2}\Bigg[ \left((3-A)A^2q x_n +  A^3q^2 z_n \right)^2 \nonumber\\
%&\quad+ (q-1) \left(  \frac{A^2q }{q-1}(q-3+A)x_n   - \frac{A^3q^2}{q-1}z_n  \right)^2 -q  A^4q^2 x_n^2 \Bigg] \nonumber\\
%& +    \frac1{q^3}\frac{d(d-1)}{2}\Bigg[     \left( (3-A)A^2q x_n +  A^3q^2 z_n  \right)^2  + (q-1) \left( \frac{A^2q(q-3+A)}{q-1} x_n - \frac{A^3q^2}{q-1}z_n\right)^2  \nonumber\\
%& +  2(q-1) \left( \frac{A^2q (q-3+A)}{q-1}x_n   - \frac{A^3q^2}{q-1}z_n \right)^2  \nonumber\\
%& +(q-1)(q-2) \left( \frac{(3q-6-2A)A^2 q}{(q-1)(q-2)} x_n +\frac{2 A^3 q^2}{(q-1)(q-2)} z_n \right)^2  \nonumber\\
%&- 2q (A^2q x_n)^2 -2q (q-1) (\frac{A^2q}{q-1} x_n)^2 \Bigg] - R\nonumber\\
%& \geq \frac1q + d A^2  x_n + \frac{d(d-1)}{2}A^4 q^2 x_n^2 + \frac{d(d-1)}{2}  \frac{q(q-4)}{q-1}  A^4 x_n^2 \nonumber\\
%& - C'  \frac{d(d-1)}{2}A^5 \left|z_n - \frac1q x_n\right|  x_n - R
%\end{align}
where $|R| \leq Cx_n^3$ and $C$ and $C'$ depend only on $q$.   Let $\kappa = \frac{q(q-4)}{3C'(q-1)}$ then if $|\lambda| \leq \kappa$ then since $0\leq z_n \leq x_n$,
\begin{align}\label{e:xnIncreasingA}
C'  \frac{d(d-1)}{2}\lambda^5 \left|\frac{z_n}{x_n} - \frac1q \right|  x_n^2
\leq C'  \kappa \lambda^4 \left|\frac{z_n}{x_n} - \frac1q \right|  x_n^2\nonumber\\
 \leq \frac13 \frac{d(d-1)}{2}  \frac{q(q-4)}{q-1}   \lambda^4 x_n^2
\end{align}
When $d>\kappa^{-2}$ then we always have $|\lambda| < \kappa$ because $d \lambda^2 \leq 1$.  For the finite number of cases when $d\leq \kappa^2$ by taking $\delta$ to be sufficiently small and $N$ to be sufficiently large we may assume by Corollary \ref{c:boundzn} that when $|\lambda|>\kappa$ and $n>N$ then $$\left|\frac{z_n}{x_n} - \frac1q \right|<\kappa. $$  It follows that  we may take equation \eqref{e:xnIncreasingA} to hold for all $d$ and $\lambda$.

Now combining equations \eqref{e:mainEquation}, \eqref{e:interchangeIntegrationusingConcen}, \eqref{e:expandApproximate} and \eqref{e:xnIncreasingA} and taking $\delta$ and $\epsilon$ to be sufficiently small and $N$ sufficiently large we complete the result.

\end{proof}

\begin{proof} (Theorem \ref{t:q5recon})

We will prove the result for the ferromagnetic case, the anti-ferromagnetic case will follow similarly.  We will establish that when $\lambda$ is close enough to $d^{-1/2}$ then $x_n$ does not converge to 0.  First we will verify that $x_n$ does not drop from a very large value to a very small one.  
Fix some $\kappa< d^{-1/2}$.
By Lemma  \ref{l:constantFactorDecrease} there exists $0<\gamma<1$ such that if $\kappa<\lambda \leq d^{-1/2}$ then $x_{n+1} \geq \gamma x_n$. 
Now we use Lemma \ref{l:xnIncreasing}. We can take $\delta>0$ and $N$ so that if $n\geq N$ and $x_n<\delta$ then
\begin{equation}\label{e:q5smallxn}
x_{n+1} \geq d \lambda^2  x_n + \frac12 \frac{d(d-1)}{2}  \frac{q(q-4)}{q-1}  \lambda^4 x_n^2 .
\end{equation}
Let $\epsilon=\min\{\frac12\gamma^{N+1},\delta\gamma\}>0$.  Since $q-4>0$ we can choose $\kappa<\lambda<d^{-1/2}$ such that
\begin{equation}\label{e:q5reconA}
1\leq d \lambda^2 + \frac12  \frac{d(d-1)}{2}  \frac{q(q-4)}{q-1}   \lambda^4 \epsilon.
\end{equation}
We now show by induction that for all $n$ that $x_n \geq \epsilon$.  Since $x_0=1-\frac1q >\frac12$, then $x_n \geq \frac12 \gamma^n \geq \epsilon$ when $n\leq N$ so suppose that $n>N$.  Now if $x_n \geq \epsilon\gamma^{-1}$ then $x_{n+1}\geq \gamma x_n \geq \epsilon$.  If $\epsilon \leq x_n\leq \gamma^{-1}\epsilon \leq \delta$ then by Lemma \ref{l:xnIncreasing} and equation \eqref{e:q5reconA} we have that,
\begin{align*}
x_{n+1} &\geq  d \lambda^2  x_n + \frac12 \frac{d(d-1)}{2}  \frac{q(q-4)}{q-1}  \lambda^4 x_n^2\\
&\geq x_n\left( d \lambda^2  + \frac12 \frac{d(d-1)}{2}  \frac{q(q-4)}{q-1}  \lambda^4 \epsilon \right)\\
&\geq x_n.
\end{align*}
It follows by induction that for all $n$, $x_n\geq \epsilon$ which implies that $\lambda^+ \leq \lambda < d^{-1/2}$ which establishes that the Kesten-Stigum bound is not tight.

\end{proof}

%\begin{proof} (Theorem \ref{t:q5reconRobust})

%At the Kesten-Stigum bound $d A^2=1$ so by Lemma \ref{l:xnIncreasing} when $n>N$ if $x_n<\delta$ then $x_{n+1} > x_n$.  By Lemma \ref{l:constantFactorDecrease} we have that $x_{n+1}\geq \gamma x_n$ and so
%\[
%\limsup x_n > 0
%\]
%which implies robust reconstruction.
%\end{proof}

\section{Large degree asymptotics}

In this section we will analyse what happens as we let $d$ grow.  As  $d$ increases the interactions become weaker and $\lambda$ decreases.  We will paramterize the interaction strengths with $\hat{\lambda}$ defined by $\hat{\lambda}=\hat{\lambda}(d)= \lambda d^{1/2}$.  With this parameterisation $\lh=1$ corresponds to the Kesten-Stigum bound in the ferromagnetic case while $\lh=-1$ corresponds to the Kesten-Stigum bound in the antiferromagnetic case.  We will, therefore, restrict our attention to $|\lh| \leq 1$.  We define
\[
U_{ij} = \log \left(1+\lambda q(Y_{ij}-\frac1q) \right).
\]
and denote $U_j=(U_{1j},\ldots,U_{qj}) \in \mathbb{R}^q$.  We have the following estimates on the means and covariances of the $U_{ij}$.
\begin{lemma}\label{l:largeDeltaIdentities}
There exists constants $C$ and $d'$ depending only on $q$ such that when $d>d'$,
\begin{equation}\label{e:largeDeltaIdentitiesA}
\left| d EU_{1j} - \frac12  \lh^2  q x_n\right| \leq Cd^{-1/2},
\end{equation}
and for $i\geq 2$,
\begin{equation}\label{e:largeDeltaIdentitiesB}
\left| d EU_{ij}  + ( \frac12  +  \frac1{q-1}) \lh^2 q x_n\right| \leq Cd^{-1/2}.
\end{equation}
For any $1\leq i \leq q$,
\begin{equation}\label{e:largeDeltaIdentitiesC}
\left| d \hbox{Var}(U_{i}) -  \lh^2  q x_n\right| \leq Cd^{-1/2}.
\end{equation}
and for and $1\leq i_1 < i_2 \leq q$,
\begin{equation}\label{e:largeDeltaIdentitiesD}
\left| d \hbox{Cov}(U_{i_1j},U_{i_2j}) + \frac1{q-1}  \lh^2  q x_n\right| \leq Cd^{-1/2}.
\end{equation}
\end{lemma}

\begin{proof}
Using the Taylor series expansion of $\log(1+w)$, there exists a constant $W>0$ such that when $|w|<W$ then $|\log(1+w) -w+\frac12 w^2| \leq |w|^3$.  Since by definition $0\leq Y_{ij}\leq 1$ by taking $d'$ to be sufficiently large we may assume that $|\lambda q(Y_{ij}-\frac1q)|\leq |\lambda|q \leq W$ since $|\lambda|\leq d^{-1/2}$.  Then by Lemma \ref{l:identities},
\begin{align}\label{e:logApproxA}
E\left| U_{1j} - \lambda q(Y_{ij}-\frac1q) +\frac12 \lambda^2 q^2 (Y_{ij}-\frac1q)^2 \right| &\leq E|\lambda|^3 q^3 |Y_{ij}-\frac1q|^3\nonumber\\
&\leq d^{-3/2} q^3 E |Y_{ij}-\frac1q|^3 \nonumber\\
&\leq q^3 d^{-3/2}.
\end{align}
Now since by Lemma \ref{l:changeOfMeasure}, $0\leq z_n \leq x_n \leq 1$ and applying the identities of Lemma \ref{l:identities},
\begin{align}\label{e:logApproxB}
&\left| E\lambda q(Y_{ij}-\frac1q) - E\frac12 \lambda^2 q^2 (Y_{ij}-\frac1q)^2 -  \frac12  \lambda^2  q x_n\right | \nonumber \\
&=\left| \lambda^2qx_n -  \frac12 \lambda^2 q^2 \left(\lambda z_n + \frac1q (1-\lambda) x_n\right)  -  \frac12  \lambda^2  q x_n\right| \nonumber\\
&=\frac12 |\lambda|^3 q^2\left|  z_n - \frac1q  x_n  \right| \nonumber\\
&\leq \frac12 q^2 d^{-\frac32}.
\end{align}
Combining equation \eqref{e:logApproxA} and \eqref{e:logApproxB} establishes equation \eqref{e:largeDeltaIdentitiesA}.  Equations \eqref{e:largeDeltaIdentitiesB}, \eqref{e:largeDeltaIdentitiesC} and \eqref{e:largeDeltaIdentitiesD} follow similarly.
\end{proof}

Since the random vectors $Y_j=(Y_{1j},\ldots,Y_{qj})$ are independent and identically distributed so are the $U_j=(U_{1j},\ldots,U_{qj})$ for $j=1,\ldots,d$.   Also each $U_{ij}$ satisfies $$|U_{ij}| \leq \max\{\log(1+d^{-1/2}q),|\log(1-d^{-1/2}q)|\}\rightarrow 0$$ as $d\rightarrow\infty$.  Such a collection of random vectors suggests the use of a central limit theorem.

The following standard proposition can be establshed using the Central Limit Theorem and Gaussian approximation.
\begin{proposition}\label{p:clt}
Let $\psi:\mathbb{R}^q\mapsto \mathbb{R}$ be a differentiable bounded function and let $\epsilon>0$.  Let $V_1\ldots,V_D$ be a sequence of iid $q$-dimensional vectors denoted $V_j=(V_{1j},\ldots, V_{qj})$.  Let $\mu\in \mathbb{R}^q$ be a vector and let $\Sigma\in \mathbb{R}^{q\times q}$ be a positive semi-definite symmetric $q\times q$-matrix.  Let $(W_1,\ldots,W_q)$ be distributed according to the $q$-dimensional Gaussian vector $N(\mu,\Sigma)$.

Suppose there exists some $C>0$ such that for $1\leq i<j \leq q$ the following holds: $\| \mu_i \|_\infty \leq C$, $\|\Sigma_{ij}\|_\infty \leq C$, $\|\mu - D E V_1\|_\infty \leq CD^{-1/2}$ and $\|\Sigma - D\hbox{Cov}(V_1) \|_\infty \leq CD^{-1/2}$ and $\| \cdot \|_\infty$ denotes the standard $L^{\infty}$ norm.    Then there exists a $D'$ depending only on $q, C$ and $\psi$ such that if $D>D'$ then
\[
\left| \psi(\sum_{i=1}^q V_{1j},\ldots, \sum_{i=1}^q V_{qj}) - \psi(W_1,\ldots,W_q) \right | \leq \epsilon
\] 
\end{proposition}

%
%\begin{proposition}\label{p:clt}
%Let $\psi:\mathbb{R}^q\mapsto \mathbb{R}$ be a differentiable bounded function and let $\epsilon>0$.  Let $V_1\ldots,V_D$ be a sequence of iid $q$-dimensional vectors denoted $V_j=(V_{1j},\ldots, V_{qj})$. Suppose there exists some $C>0$ such that for $1\leq i \leq q$, $D|E V_{ij}|\leq C$, $D\hbox{Var}(V_{ij}) \leq C$ and $|V_{ij}|\leq C D^{-1/2}$.  Let $\mu = D E V_1$ and $\Sigma = D\hbox{Cov}(V_1)$ and let $(W_1,\ldots,W_q)$ be distributed according to the $q$-dimensional Gaussian vector $N(\mu,\Sigma)$.  Then there exists a $D'$ depending only on $q, C$ and $\psi$ such that if $D>D'$ then
%\[
%\left| \psi(\sum_{i=1}^q V_{1j},\ldots, \sum_{i=1}^q V_{qj}) - \psi(W_1,\ldots,W_q) \right | \leq \epsilon
%\] 
%\end{proposition}

%Using the mean and covariance structure we established in Lemma \ref{l:largeDeltaIdentities} we can hence apply the Central Limit Theorem which shows that,
%\[
%\sum_{i=1}^d U_j \rightarrow N(\lambda^2 x_n \mu, \lambda^2 x_n \Sigma)
%\]
%where the convergence is in distribution to the q-dimensional normal distribution and 

Let $\mu$ be the $q$-dimensional vector given by 
\[
\mu_{i} = \begin{cases} \frac{q}{2}  & i=1, \\
-q( \frac12  +  \frac1{q-1})& i\neq 2, \end{cases}
\]
and let $\Sigma$ is the $q\times q$-covariance matrix given by
\[
\Sigma_{ij} = \begin{cases} q  & i=j, \\
- \frac{q}{q-1}& i\neq j. \end{cases}
\]
Define 
\[
\psi(w_1,\ldots,w_q)=\frac{e^{w_1}}{\sum_{i=1}^q e^{w_i}}.
\]
The function $\psi$ is positive, analytic and bounded by 1.  Now if $(W_1,\ldots,W_q)$ is a Gaussian vector distributed according to 
$N(0,\Sigma)$ then $(s \mu_1 + \sqrt{s}W_1,\ldots,s \mu_q + \sqrt{s} W_q)$ is distributed according to $N(s \mu, s \Sigma)$.  We define
\begin{align}
g(s)= g_q(s) &= E \psi(s \mu_1 + \sqrt{s}W_1,\ldots,s \mu_q + \sqrt{s} W_q) -\frac1q\nonumber\\
&=\frac{e^{s \mu_1 + \sqrt{s}W_1}}{\sum_{i=1}^q e^{s \mu_i + \sqrt{s}W_i}}-\frac1q.
\end{align}
Since $Z_i = \exp( \sum_{i=1}^q U_{ij} )$ we have that
\[
x_{n+1} = E \frac{Z_1}{\sum_{i=1}^q Z_i} -\frac1q = E\psi(\sum_{j=1}^d U_{1j},\ldots, \sum_{j=1}^d U_{qj})-\frac1q.
\]
Then Proposition \ref{p:clt} and Lemma \ref{l:largeDeltaIdentities} immediately imply the following lemma.
\begin{lemma}\label{l:clt}
For each $\epsilon>0$ there exists a $d'$ such that when $d>d'$,
\[
\left | x_{n+1} - g(\hat{\lambda}^2 x_n) \right| \leq \epsilon.
\]
\end{lemma}

%With an identical proof Lemma \ref{l:clt} also applies to the robust reconstruction setting.
Understanding the function $g_q(s)$, and in particular the solutions to the equation $g_q(s)=s$, provides key information into the reconstruction problem when $d$ is large.  Since $0<x_n \leq \frac{q-1}{q}$ we will restrict our attention on $g$ to this interval.

\begin{lemma}\label{l:gDifferentiable}
For each $q$, the function $g_q$ is continuously differentiable on the interval $(0,\frac{q-1}{q}]$ and increasing.
\end{lemma}

\begin{proof}
Since
\begin{equation}\label{e:logisticDerivative}
\sup_x\left |\frac{d}{dx} \frac{e^x}{1+e^x} \right| =  \sup_x \left | \frac{e^x}{(1+e^x)^2} \right| = \frac14
\end{equation}
we have that when $s>0$,
\begin{align*}
E\left |\frac{d}{ds} \psi(s \mu_1 + \sqrt{s}W_1,\ldots,s \mu_q + \sqrt{s} W_q)  \right| &\leq \frac14 E \sum_{i=1}^q\left| \frac{d}{ds} s \mu_i + \sqrt{s}W_i \right|<\infty
\end{align*}
which establishes that $g_q$ is differentiable.  Now let $(\widetilde{W}_1,\widetilde{W}_2,\ldots,\widetilde{W}_q)$ be an independent copy of $(W_1,\ldots,W_q)$.  Then when $0\leq s'<s$ the following equality in distribution holds
\begin{align*}
\sqrt{s}\left(W_1,\ldots,W_q\right) &\stackrel{d}{=}  \sqrt{s'} \left(W_1,\ldots,W_q\right)\\
& + \sqrt{s-s'}\left(\widetilde{W}_1,\widetilde{W}_2,\ldots,\widetilde{W}_q\right).
\end{align*}
Recall that if $W$ is distributed as $N(\mu,s^2)$ then $Ee^W= e^{\mu+\frac12 s^2}$.  For $2\leq i \leq q$, since $\widetilde{W}_i-\widetilde{W}_1$ is distributed as $N(0,2q+\frac{2q}{q-1})$,
\begin{align*}
&E\left[ \exp\left(\sqrt{s'}(W_i-W_1) + \sqrt{s-s'}(\widetilde{W}_i-\widetilde{W}_1) \right)\mid \{W\}_{j=1}^q \right] \\
=& \exp\left(\sqrt{s'}(W_i-W_1) + (s-s') (q+\frac{q}{q-1})\right).
\end{align*}
Noting  that $\frac1{1+u}$ is convex, by Jensen's inequality
\begin{align*}
g_q(s) &= E \psi(s \mu_1 + \sqrt{s}W_1,\ldots,s \mu_q + \sqrt{s} W_q) -\frac1q\\
&= E \frac1{1+ \sum_{i=2}^q \exp\left(-s\left(q+\frac{q}{q-1}\right)+ \sqrt{s'}(W_i-W_1) + \sqrt{s-s'}(\widetilde{W}_i-\widetilde{W}_1) \right)}-\frac1q\\
&\geq E \frac1{1+ E\left[ \sum_{i=2}^q \exp\left(-s\left(q+\frac{q}{q-1}\right)+ \sqrt{s'}(W_i-W_1) + \sqrt{s-s'}(\widetilde{W}_i-\widetilde{W}_1) \right)\mid   \{W\}_{j=1}^q   \right]}-\frac1q\\
&= E \frac1{1+  \sum_{i=2}^q \exp\left(-s'\left(q+\frac{q}{q-1}\right)+ \sqrt{s'}(W_i-W_1)  \right)}-\frac1q\\
&= g_q(s')
\end{align*}
which establishes that $g_q(s)$ is increasing.
\end{proof}

\begin{lemma}\label{l:smalls}
For all $q$ and small $s$, we have that
\begin{equation}\label{e:smallsResult}
g_q(s) = s + \frac12 \frac{(q-4)q}{q-1} s^2 +\frac16 \frac{(q^2-18q+42)q^2}{(q-1)^2} s^3 + O(s^4)
\end{equation}
and so when $q\geq 5$ there is a root $0<s^* < \frac{q-1}{q}$ to the equation $g(s^*)=s^*$.
\end{lemma}

\begin{proof}
Using the identity
\[
\frac{a}{r+s} = \left(\sum_{i=1}^m (-1)^{i-1}\frac{a r^{i-1}}{s^i} \right) +(-1)^m \frac{r^m}{s^m}\frac{a}{r+s}
\]
and taking $a= \exp(s \mu_1 + \sqrt{s}W_1), s=q$ and $r= \left( \sum_{i=1}^q \exp(s \mu_i + \sqrt{s}W_i) -q \right)$ we have that
\begin{align}\label{e:smallsA}
g_q(s) &= E \psi(s \mu_1 + \sqrt{s}W_1,\ldots,s \mu_q + \sqrt{s} W_q) -\frac1q\nonumber\\
&=E \sum_{i=1}^4 (-1)^{i-1}\frac{ \left( \sum_{i=1}^q \exp(s \mu_i + \sqrt{s}W_i) -q \right)^{i-1}  \exp(s \mu_1 + \sqrt{s}W_1)}{q^i}\nonumber \\
&\quad + E  \frac{ \left( \sum_{i=1}^q \exp(s \mu_i + \sqrt{s}W_i) -q \right)^4}{q^4}   \frac{\exp(s \mu_1 + \sqrt{s}W_1)}{ \sum_{i=1}^q \exp(s \mu_i + \sqrt{s}W_i) }  - \frac1q.
\end{align}
Now again using the fact that if $W$ is distributed as $N(\mu,s^2)$ then $E e^W=e^{\mu+s^2/2}$ and doing Taylor series expansions with the help of Mathematica we have that
\begin{align*}%\label{e:smallsB}
&E \sum_{i=1}^4 (-1)^{i-1}\frac{ \left( \sum_{i=1}^q \exp(s \mu_i + \sqrt{s}W_i) -q \right)^{i-1}  \exp(s \mu_1 + \sqrt{s}W_1)}{q^i}\nonumber \\
&= \bigg( 4\,q{e^{6\,qs}}+6\,{e^{{\frac {qs \left( q-10 \right) }{q-1}}}
}-{e^{10\,qs}}+8\,{e^{3\,{\frac { \left( q-2 \right) sq}{q-1}}}}{q}^{2
}-3\,{e^{2\,{\frac {qs \left( 3\,q-5 \right) }{q-1}}}}q+3\,{e^{2\,{
\frac {qs \left( 3\,q-5 \right) }{q-1}}}}\\
&-6\,{e^{2\,{\frac {qs \left(
q-5 \right) }{q-1}}}}-{q}^{3}-6\,{q}^{2}{e^{3\,qs}}+4\,{e^{2\,{\frac {
qs \left( 2\,q-5 \right) }{q-1}}}}-11\,{e^{{\frac {qs \left( q-10
 \right) }{q-1}}}}q-12\,{e^{{\frac {qs \left( q-6 \right) }{q-1}}}}{q}
^{2}\\
& -{e^{{\frac {qs \left( q-10 \right) }{q-1}}}}{q}^{3}+4\,{e^{2\,{
\frac {qs \left( q-3 \right) }{q-1}}}}{q}^{2}-4\,{e^{2\,{\frac {qs
 \left( q-3 \right) }{q-1}}}}q +4\,{e^{{\frac {qs \left( q-6 \right) }{
q-1}}}}{q}^{3}+8\,{e^{{\frac {qs \left( q-6 \right) }{q-1}}}}q\\
& -4\,{e^{
2\,{\frac {qs \left( 2\,q-5 \right) }{q-1}}}}q-3\,{e^{{\frac {qs
 \left( -10+3\,q \right) }{q-1}}}}{q}^{2}-3\,{e^{2\,{\frac {qs \left(
q-5 \right) }{q-1}}}}{q}^{2}+9\,{e^{2\,{\frac {qs \left( q-5 \right) }
{q-1}}}}q+6\,{e^{{\frac {qs \left( q-3 \right) }{q-1}}}}{q}^{2}\\
&-6\,{q}^{3}{e^{{\frac {qs \left( q-3 \right) }{q-1}}}}-6\,{e^{{\frac {qs
 \left( -10+3\,q \right) }{q-1}}}}+4\,{q}^{3}{e^{qs}}-8\,{e^{3\,{
\frac { \left( q-2 \right) sq}{q-1}}}}q+6\,{e^{{\frac {qs \left( q-10
 \right) }{q-1}}}}{q}^{2}+9\,{e^{{\frac {qs \left( -10+3\,q \right) }{
q-1}}}}q \bigg) {q}^{-4}\\
&= \frac1q+s+\frac12\,{\frac { \left( q-4 \right) q}{q-1}}{s}^{2}+\frac16\,{
\frac { \left( {q}^{2}-18\,q+42 \right) {q}^{2}}{ \left( q-1 \right) ^
{2}}}{s}^{3} + O \left( {s}^{4}  \right) )
\end{align*}
and
\begin{align*}%\label{e:smallsC}
&E  \frac{ \left( \sum_{i=1}^q \exp(s \mu_i + \sqrt{s}W_i) -q \right)^4}{q^4} \\
&= - \bigg( 4\,q{e^{6\,qs}}+60\,{e^{{\frac {qs \left( q-10 \right) }{q-1}
}}}-{e^{10\,qs}}+16\,{e^{3\,{\frac { \left( q-2 \right) sq}{q-1}}}}{q}
^{2}-5\,{e^{2\,{\frac {qs \left( 3\,q-5 \right) }{q-1}}}}q\\
&+4\,{e^{-6\,
{\frac {qs}{q-1}}}}{q}^{4}-6\,{q}^{4}{e^{-3\,{\frac {qs}{q-1}}}} -12\,{
e^{-3\,{\frac {qs}{q-1}}}}{q}^{2}-{e^{-10\,{\frac {qs}{q-1}}}}{q}^{4}-
35\,{e^{-10\,{\frac {qs}{q-1}}}}{q}^{2}\\
& -24\,{e^{-6\,{\frac {qs}{q-1}}}
}q+50\,{e^{-10\,{\frac {qs}{q-1}}}}q+44\,{e^{-6\,{\frac {qs}{q-1}}}}{q
}^{2}+5\,{e^{2\,{\frac {qs \left( 3\,q-5 \right) }{q-1}}}} +10\,{e^{-10
\,{\frac {qs}{q-1}}}}{q}^{3}\\
&-30\,{e^{2\,{\frac {qs \left( q-5 \right)
}{q-1}}}}-{q}^{4}-6\,{q}^{2}{e^{3\,qs}}+4\,{q}^{4}{e^{-{\frac {qs}{q-1
}}}}+10\,{e^{2\,{\frac {qs \left( 2\,q-5 \right) }{q-1}}}}-110\,{e^{{
\frac {qs \left( q-10 \right) }{q-1}}}}q\\
&-72\,{e^{{\frac {qs \left( q-6
 \right) }{q-1}}}}{q}^{2}-10\,{e^{{\frac {qs \left( q-10 \right) }{q-1
}}}}{q}^{3}+12\,{e^{2\,{\frac {qs \left( q-3 \right) }{q-1}}}}{q}^{2}-
12\,{e^{2\,{\frac {qs \left( q-3 \right) }{q-1}}}}q+24\,{e^{{\frac {qs
 \left( q-6 \right) }{q-1}}}}{q}^{3}\\
 &+48\,{e^{{\frac {qs \left( q-6
 \right) }{q-1}}}}q-10\,{e^{2\,{\frac {qs \left( 2\,q-5 \right) }{q-1}
}}}q-10\,{e^{{\frac {qs \left( -10+3\,q \right) }{q-1}}}}{q}^{2}-15\,{
e^{2\,{\frac {qs \left( q-5 \right) }{q-1}}}}{q}^{2}+45\,{e^{2\,{
\frac {qs \left( q-5 \right) }{q-1}}}}q\\
&+18\,{e^{{\frac {qs \left( q-3
 \right) }{q-1}}}}{q}^{2}-18\,{q}^{3}{e^{{\frac {qs \left( q-3
 \right) }{q-1}}}}-24\,{e^{-10\,{\frac {qs}{q-1}}}}-24\,{e^{-6\,{
\frac {qs}{q-1}}}}{q}^{3}+18\,{q}^{3}{e^{-3\,{\frac {qs}{q-1}}}}\\
&-20\,{
e^{{\frac {qs \left( -10+3\,q \right) }{q-1}}}}-4\,{q}^{3}{e^{-{\frac
{qs}{q-1}}}}+4\,{q}^{3}{e^{qs}}-16\,{e^{3\,{\frac { \left( q-2
 \right) sq}{q-1}}}}q\\
 &+60\,{e^{{\frac {qs \left( q-10 \right) }{q-1}}}}
{q}^{2}+30\,{e^{{\frac {qs \left( -10+3\,q \right) }{q-1}}}}q \bigg)
{q}^{-4}\\
&=O(s^4).
\end{align*}
Since $0 \leq \frac{ \left( \sum_{i=1}^q \exp(s \mu_i + \sqrt{s}W_i) -q \right)^4}{q^4}$ and $0 \leq \frac{\exp(s \mu_1 + \sqrt{s}W_1)}{ \sum_{i=1}^q \exp(s \mu_i + \sqrt{s}W_i) }\leq 1$ combining these estimates establishes equation \eqref{e:smallsResult}.

Since $q-4 >0$ when $q\geq 5$ for small $s>0$ we have that $g_q(s)>s$.  Since 
\[
g_q(1 - \frac1{q}) = E \psi(s \mu_1 + \sqrt{s}W_1,\ldots,s \mu_q + \sqrt{s} W_q) -\frac1q < 1- \frac1q
\]
by the Intermediate Value Theorem there must be some $0<s^*<\frac{q-1}{q}$ such that $g(s^*)=s^*$.

\end{proof}

\begin{theorem}\label{l:asymQ5}
When $q\geq5$ define
\[
w^* = \inf \{w: \exists 0<s^*<\frac{q-1}{q}, g(w s^*)=s^*\}.
\]
Then $0<w^*<1$ and for each $\delta>0$ there exists a $d'(q,\delta)$ such that if $d>d'$ then the model has reconstruction when $\lh^2 \geq w^*+\delta$ but does not have reconstruction when $\lh^2 \leq w^*-\delta$.
\end{theorem}

\begin{proof}
The key idea of this result is that when $\lh^2 > w^*$, $g_q(\lh s)$ has a non-zero attractive fixed point as a function of $s$ while if $\lh< w^*$ then $g_q(\lh s) < s$ for $s>0.$
By Lemma \ref{l:smalls} we have the expansion $g_q(s) = s+  \frac12 \frac{(q-4)q}{q-1} s^2 + o(s^2)$ so for small $s$, $g_q(s)>s$.  It also implies that for any $0<w<1$, the set $\{0<s<\frac{q-1}{q}: g_q(w s)\geq s\}$ is a compact set bounded away from 0.  By the continuity of $g_q$,
\[
\left \{0<s<\frac{q-1}{q}: g(w^* s)= s \right\} = \bigcap_{w^*<w<1} \left \{0<s<\frac{q-1}{q}: g(w s)\geq s \right\}
\]
and by the Finite Intersection Property of compact sets it is nonempty and compact so let $s^*\in\{0<s<\frac{q-1}{q}: g(w^* s)= s\}$.

Now set $\lh^2 = w^*+\delta$ and so $$g_q((w^*+\delta)(s^* \frac{w^*}{w^*+\delta}))=g_q(s^* w^*)=s^* > s^* \frac{w^*}{w^*+\delta}.$$  Take $d$ large enough so that Lemma \ref{l:clt} holds with $0<\epsilon< s^* - s^* \frac{w^*}{w^*+\delta}$.  Then when $x_{n} >  s^* \frac{w^*}{w^*+\delta}$ since $g_q$ is monotone it follows that 
\begin{align*}
x_{n+1} &\geq g_q((w^*+\delta)x_n)-\epsilon \\
&> g_q((w^*+\delta)(s^* \frac{w^*}{w^*+\delta})) - (s^* - s^* \frac{w^*}{w^*+\delta})\\
& =  s^* \frac{w^*}{w^*+\delta}
\end{align*}
and hence $\inf x_n \geq s^* \frac{w^*}{w^*+\delta}$ which establishes reconstruction.

%The function $g_q$ is $C^\infty$ on $(0,\frac{q-1}{q}$ and so since $g_q(w^* s)\leq s$ we have
%$\frac{d}{ds} g_q(w^* s)\big|_{s=s^*} = 1$ and so for some $\epsilon>0$ we have that $g_q(w^* s)$ is increasing for $s$ in $[s^*,s^*+\epsilon]$.  Since $g(w^* (s^*+\epsilon))< s^*+\epsilon$ we may assume that $\delta$ is small enough so that $g_q(w^* (s^*+\epsilon))< \frac{w^*}{w^* + \delta} (s^*+\epsilon)$.  It follows that all $s\in [ \frac{w^*}{w^* + \delta} s^*, \frac{w^*}{w^* + \delta} (s^*+\epsilon)]$,
%\[
% \frac{w^*}{w^* + \delta} s^* < g((w^* + \delta)s) <  \frac{w^*}{w^* + \delta} (s^* + \epsilon).
%\]
%It follows that if $\lambda^2=w^* + \delta$ and $x_n \in  [ \frac{w^*}{w^* + \delta} s^*, \frac{w^*}{w^* + \delta} (s^*+\epsilon)]$ then when $d$ is sufficiently large by Lemma \ref{l:clt}, $x_{n+1} \in [ \frac{w^*}{w^* + \delta} s^*, \frac{w^*}{w^* + \delta} (s^*+\epsilon)]$.  

%Suppose for the moment suppose that we our data was $\frac{w^*}{w^* + \delta} s^* \frac{q-1}{q}$-censored.  Then we would have that $x_0=\frac{w^*}{w^* + \delta} s^*$.  The calculations above then show that $\inf x_n \geq \frac{w^*}{w^* + \delta} s^*$ and so have reconstruction.  Now as reconstruction for uncensored data implies reconstruction for uncensored date (because you can sensor it) we have reconstruction when $\lambda=w^*+\delta$.

By equation \eqref{e:crudeExpansion}
\[
\left| x_{n+1} -\lh^2  x_n\right| \leq C_q \lambda^4 \frac{d(d-1)}{2}x_n^2 \leq C_q x_n^2
\]
where $C_q$ does not depend on $d$ or $\lh$.  So when $|\lh|<1$ and if $x_n < \frac{1-\lh^2}{2 C_q}$ then $$x_{n+1}\leq \lh^2 x_n + C_q x_n^2\leq \lh^2 x_n + \frac{1-\lh^2}{2} x_n < \frac{1+\lh^2}{2} x_n.$$  When $\lh^2 <w^*$ then $g(\lh^2 s) \leq \frac{\lh^2}{w^*} s$ and so by Lemma \ref{l:clt} for large enough $d$, we have that for some $n$, $x_n < \frac{1-\lh^2}{2 C_q}$.  It follows then that
$x_n$ converges to 0 which proves non-reconstruction for large enough $d$.
\end{proof}

%
%\begin{lemma}
%With $q=4$, for each $|\lambda| < 1$ there is a $d'$ such that if $d> d'$ then we have non-reconstruction.
%\end{lemma}
%Applying Lemma \ref{l:gDecreasing34} the proof follows similarly to the proof of Lemma \ref{l:asymQ5}.

\subsection{Non-reconstruction for $q=3$}

\begin{lemma}\label{l:gDecreasing34}
When $q=3$ for all $0\leq s \leq \frac{q-1}{q}$ then $g_q(s) < s$.

\end{lemma}
We defer this proof to the appendix.

\begin{lemma}\label{l:xnDecreasing}
When $q=3$ there exists a $\delta>0$ and $N$ not depending on $d$ or $\lambda$ such that if $x_n\leq \delta$ and $n>N$ then
\[
x_{n+1} \leq d \lambda^2  x_n -  \frac{3}{4} \frac{d(d-1)}{2}    \lambda^4 x_n^2 .
\]
\end{lemma}

The proof is essentially identical to the proof of Lemma \ref{l:xnIncreasing} and so we omit it.

\begin{proof}(Theorem \ref{t:q3nonrecon})

At the Kesten-Stigum bound we have that $|\lh|=1$.  Since $g(s)<s$ for all $s>0$ by Lemma \ref{l:clt} there exists a $d'$ such that when $d>d'$ and $m$ is sufficiently large then $x_m < \delta$ where $\delta$ is the constant in Lemma \ref{l:xnDecreasing}. It follows from Lemma  \ref{l:xnDecreasing} that if for some $m$, $x_m < \delta$ then $\lim_n x_n=0$ and hence non-reconstruction.

\end{proof}

%\begin{proof}(Theorem \ref{t:q3nonreconRobust})

%

%
%\end{proof}

\textbf{Acknowledgements}  AS would like to thank Elchanan Mossel for his encouragement, insightful discussions and careful reading of a draft of this paper.

\bibliographystyle{plain}
\bibliography{allbib}

\appendix
\section{Deferred Proof}

\begin{proof}(Lemma \ref{l:gDecreasing34})

Recall that  $\mu$ is the $q$-dimensional vector given by 
\[
\mu_{i} = \begin{cases} \frac{q}{2}  & i=1, \\
-q( \frac12  +  \frac1{q-1})& i\neq 2, \end{cases}
\]
and that $\Sigma$ is the $q\times q$-covariance matrix given by
\[
\Sigma_{ij} = \begin{cases} q  & i=j, \\
- \frac{q}{q-1}& i\neq j. \end{cases}
\]
With $(W_1,\ldots,W_q)$ a Gaussian vector distributed according to 
$N(0,\Sigma)$ the function $g_q(s)$ is defined as
\[
g_q(s) = E \psi(s \mu_1 + \sqrt{s}W_1,\ldots,s \mu_q + \sqrt{s} W_q) -\frac1q.
\]
where
\[
\psi(w_1,\ldots,w_q)=\frac{e^{w_1}}{\sum_{i=1}^q e^{w_i}}.
\]
In this lemma we consider the case of $q=3$. By equation \eqref{e:logisticDerivative} we have that for any $x,y$,
\[
\left| \frac{e^x}{1+e^x} - \frac{e^y}{1+e^y} \right| \leq \frac14 |x-y|,\quad \left| \frac{1}{1+e^x} - \frac{1}{1+e^y} \right| \leq \frac14 |x-y|.
\]
Using this estimate and the fact that $E|W_i|=\sqrt{\frac{6}{\pi}}$ it follows that
\begin{align*}
\left|g_3(s_1)-g_3(s_2)\right|&\leq \frac14 \sum_{i=1}^3 |\mu_i(s_1-s_2)|+ \left| \sqrt{s_1} -  \sqrt{s_2}\right  | E |W_i|  \\
&=\frac{15}{8} |s_1-s_2| + \sqrt{\frac{27}{8\pi}} \left| \sqrt{s_1} -  \sqrt{s_2}\right |.
\end{align*}
Now $\max_{x\in[0.1, \frac23]} \frac{d}{dx} x^{1/2} = \frac12 \sqrt{10}$.
Hence if we take $0.1 \leq s_1 < s_2 \leq \frac23$ then 
\begin{equation}\label{e:gDerivative}
\left| g_3(s_1)-g_3(s_2)\right| \leq  ( \frac{15}{8} + \sqrt{\frac{135}{16\pi}} ) |s_1-s_2|\leq 3 |s_1-s_2|.
\end{equation}
Let
$$\mathcal{S}=\left\{\frac{100}{1000}, \frac{101}{1000},\ldots,\frac{667}{1000}\right\}$$
and suppose that 
\begin{equation}\label{e:gInequalityGrid}
\forall s^*\in\mathcal{S} \quad g_3(s^*)-s^* < -\frac{5}{1000}.
\end{equation}
Now fix some $s\in[0.1,\frac23]$.  Then for some $s^*\in \mathcal{S}$, $|s-s^*|<\frac1{1000}$ which implies that
\begin{align*}
g_3(s)-s &\leq g_3(s^*) -  s^* + |g_3(s) - g_3(s^*)| + |s-s^*| \\
&< -\frac{5}{1000} + 4 |s - s^*| + |s-s^*| \\
& < 0
\end{align*}
 where the second inequality follows from equation \eqref{e:gDerivative}.  So proving equation \eqref{e:gInequalityGrid} would imply that $g_3(s)<s$ for all $0.1\leq s \leq \frac23$.  We do this by a rigorous method of numerical integration.  
 
Let $U_1, U_2$ be independent standard Gaussians.  The random vectors $(W_2-W_1,W_3-W_1)$ and $(3U_1,\frac32 U_1 + \frac{3 \sqrt{3}}{2} U_2)$ have the same covariance matrix and therefore are equal in distribution.  Hence
 \begin{align}\label{e:g3IntegralApprox}
 g_3(s) &= E \frac1{1+ \sum_{i=2}^3 \exp\left(-\frac{9s}{2}+ \sqrt{s}(\widetilde{W}_i-\widetilde{W}_1) \right)}-\frac13\nonumber \\
&= E  \frac1{1+ \exp\left(-\frac{9s}{2}+ 3\sqrt{s}U_1 \right)+ \exp\left(-\frac{9s}{2}+ \frac32\sqrt{s}U_1 +  \frac{3 \sqrt{3}}{2} \sqrt{s} U_2 \right)}-\frac13 \nonumber \\
&= \int_{\mathbb{R}^2} \frac1{1+ \exp\left(-\frac{9s}{2}+ 3\sqrt{s}x \right)+ \exp\left(-\frac{9s}{2}+ \frac32\sqrt{s}x +  \frac{3 \sqrt{3}}{2}  \sqrt{s} y \right)} \nonumber  \\
&\quad\quad \quad \cdot \frac{\exp(-x^2/2-y^2/2)}{2\pi} \ dx \ dy -\frac13 \nonumber \\
& \leq   \int_{-5}^5 \int_{-5}^5 \frac1{1+ \exp\left(-\frac{9s}{2}+ 3\sqrt{s}x \right)+ \exp\left(-\frac{9s}{2}+ \frac32\sqrt{s}x +  \frac{3 \sqrt{3}}{2}  \sqrt{s} y \right)} \nonumber  \\
&\quad\quad \quad \cdot \frac{\exp(-x^2/2-y^2/2)}{2\pi} \ dx \ dy -\frac13 + 10^{-5}
 \end{align}
where the inequality uses the standard inequality that
\[
\int_x^\infty \frac{\exp(-x^2/2)}{\sqrt{2\pi}} dx \leq  \frac{\exp(-x^2/2)}{x\sqrt{2\pi}}
\]
which implies that
\[
\iint_{\mathbb{R}^2\setminus [-5,5]^2} \frac{\exp(-x^2/2-y^2/2)}{2\pi} \leq 4\frac{\exp(-5^2/2)}{5\sqrt{2\pi}} \leq 10^{-5}.
\]
Define the function $\phi(i) = \min\{|i|,|i+1|\}$.  Then for integers $i$ and $j$,
 \begin{align}\label{e:integralApproximation}
&\int_{\frac{i}{200}}^{\frac{i+1}{200}}\int_{\frac{j}{200}}^{\frac{j+1}{200}}  \frac{\exp(-x^2/2-y^2/2)  \ dx \ dy }{\left(1+ \exp\left(-\frac{9s}{2}+ 3\sqrt{s}x \right)+ \exp\left(-\frac{9s}{2}+ \frac32\sqrt{s}x +  \frac{3 \sqrt{3}}{2}  \sqrt{s} y  \right) \right) 2\pi} \nonumber \\
&\leq \frac{\exp(-(\frac{\phi(i)}{200})^2/2-(\frac{\phi(j)}{200})^2/2) 40000^{-1} }{\left(1+ \exp\left(-\frac{9s}{2}+ 3\sqrt{s} \frac{i}{200} \right)+ \exp\left(-\frac{9s}{2}+ \frac32\sqrt{s} \frac{i}{200} +  \frac{3 \sqrt{3}}{2}  \sqrt{s}   \frac{j}{200}  \right) \right) 2\pi}.
 \end{align}
Let $\psi(i,j)$ denote the right hand-side of equation \eqref{e:integralApproximation}.  Substituting this bound in \eqref{e:g3IntegralApprox} we have that
\begin{equation}\label{e:g3SumApprox}
 g_3(s) 
\leq  -\frac13 + 10^{-5} + \sum_{i=-1000}^{999} \  \sum_{j=-1000}^{999} \psi(i,j).
\end{equation}
The right hand side of equation \eqref{e:g3SumApprox} is merely a combination of basic arithmetic operations and exponentials and so can be rigorously computed to arbitrarily high precision (e.g. in Mathematica).  Evaluating this expression for  each $s^*\in \mathcal{S}$ establishes equation \eqref{e:gInequalityGrid}.  As noted above this implies that $g(s) < s$ when $s\in [0.1,\frac23]$.

It remains to show that $g_3(s)<s$ when $0< s \leq 0.1$.  Using equation \eqref{e:smallsA} and noting that 
\[
\frac{\exp(s \mu_1 + \sqrt{s}W_1)}{ \sum_{i=1}^3 \exp(s \mu_i + \sqrt{s}W_i) } \leq 1
\]
we have that
\begin{align*}
g_3(s) &\leq E \sum_{i=1}^4 (-1)^{i-1}\frac{ \left( \sum_{i=1}^3 \exp(s \mu_i + \sqrt{s}W_i) -3 \right)^{i-1}  \exp(s \mu_1 + \sqrt{s}W_1)}{3^i}\nonumber \\
&\quad + E  \frac{ \left( \sum_{i=1}^3 \exp(s \mu_i + \sqrt{s}W_i) -3 \right)^4}{81}  - \frac13.
\end{align*}
Using the fact that if $W$ is distributed as $N(\mu,\sigma^2)$ then $E e^W=e^{\mu+\sigma^2/2}$
 we have after simplifying that
\begin{equation}\label{e:g3Estimate}
g_3(s) \leq \frac{74}{27} - \frac{4}{27}e^{-9s/2}+\frac{4}{27}e^{3s}-\frac{202}{81} e^{-3s/2}+\frac{8}{27}e^{-6s}+\frac{4}{81}e^{12s}-\frac{16}{27}e^{9s/2}.
\end{equation}
By Taylor's Theorem we have that if $|x| \leq 1.2$ then
\begin{align*}
\left| \exp(x) - \sum_{i=0}^5 \frac{x^i}{i!} \right| \leq  \frac{x^6}{6!} \max_{y\in [-1.2,1.2]} \left|\frac{d^6 e^y}{dy^6} \right|  \leq 2 \frac{x^6}{6!}.
\end{align*}
Applying this to equation \eqref{e:g3Estimate} we get that when $0\leq s \leq 0.1$ that
\[
g_3(s) - s \leq \frac{1}{1280}s^2 h(s)
\]
where 
\[
h(s)=-960-1440s+58860s^2+98334s^3+595795s^4.
\]
Now $h(s)$ is convex and $h(0)<0$ and $h(0.1)<0$ which imples that  $h(s)<0$ for all $0\leq s \leq 0.1$.  It follows that $g_3(s)<s$ for all $0<s\leq 0.1$ which completes the proof.

\end{proof}

\end{document}